\documentclass[iicol]{sn-jnl}

\usepackage{graphicx}%
\usepackage{multirow}%
\usepackage{amsmath,amssymb,amsfonts}%
\usepackage{amsthm}%
\usepackage{mathrsfs}%
\usepackage[title]{appendix}%
\usepackage{xcolor}%
\usepackage{textcomp}%
\usepackage{manyfoot}%
\usepackage{booktabs}%
\usepackage{algorithm}%
\usepackage{algorithmicx}%
\usepackage{algpseudocode}%
\usepackage{listings}%

\newtheorem{theorem}{Theorem}
\newtheorem{proposition}[theorem]{Proposition}%
\newtheorem{lemma}[theorem]{Lemma}%
\newtheorem{corollary}[theorem]{Corollary}%
\newtheorem{remark}{Remark}%

\usepackage[none]{hyphenat}

\usepackage{amsopn,amsmath,amssymb}
\usepackage{siunitx}
\newcommand{\ZZ}{\mathbb{Z}}
\newcommand{\RR}{\mathbb{R}}

\newcommand{\NN}{\mathbb{N}}

\newcommand{\db}{{\boldsymbol{d}}}

\newcommand{\fb}{{\boldsymbol{f}}}

\newcommand{\hb}{{\boldsymbol{h}}}

\DeclareSIUnit\byte{B}  

\newcommand{\imagewidthincolumn}{0.95\hsize}
\newcommand{\imagelesswidthincolumn}{0.8\hsize}

\newcommand{\imagelinebreak}{\\\vspace*{0.5em}}

\begin{document}

\title[Haar Wavelets and Total Variation Regularization]{Haar Wavelets, Gradients
  and Approximate Total Variation Regularization} 


\author[1,2,3]{\fnm{Tomas} \sur{Sauer}}\email{tomas.sauer@uni-passau.de} 
\equalcont{All authors have contributed equally.}

\author*[3]{\fnm{Andreas Michael}
  \sur{Stock}}\email{stock@forwiss.uni-passau.de} 
\equalcont{All authors have contributed equally.} 

\affil*[1]{\orgdiv{Chair of Mathematical Image Processing},
  \orgname{University of Passau}, \orgaddress{\street{Innstr. 43}, \city{Passau}, \postcode{94032}, \country{Germany}}}

\affil[2]{\orgdiv{Development Center X-ray Technology},
  \orgname{Fraunhofer Institute of Integrated Circuits},
  \orgaddress{\street{Innstr. 43}, \city{Passau}, \postcode{94032},
    \country{Germany}}} 

\affil[3]{\orgdiv{FORWISS},
  \orgname{University of Passau}, \orgaddress{\street{Innstr. 43},
    \city{Passau}, \postcode{94032}, \country{Germany}}}


\abstract{Image denoising by means of  total variation (TV) regularization
  is still a standard procedure. For very large images, especially
  three-dimensional voxel datasets, however, this can be computationally
  infeasible. 
  We show how this \emph{TV regularization} can be approximately
  performed even in 
  arbitrary dimensions by applying
  appropriate shrinkage to selected and properly weighted Haar wavelet
  coefficients, all of which depends even on the dimensionality of the
  data. 
  Our approach acts entirely on the wavelet coefficients which
  represent the compressed image, and is
  therefore suited for the application on large three-dimensional
  images represented in the Haar wavelet basis, e.g., volumes from
  computed tomography. }





\maketitle

\section{Introduction}
\label{sec:Introduction}
Modern imaging techniques produce larger and larger images that
need to be processed and analyzed. This is particularly evident in
industrial computed tomography (CT) where huge datasets of
one terabyte and more are 
generated on an almost routine basis by either scanning very large
objects~\cite{salamon19:_xxl_ct_elect_vehic} or working with a very
high resolution. Even on modern consumer hardware, such \emph{bigtures}
cannot be handled  without substantial compression. A natural approach
for that purpose is to use well-established tensor product wavelet
methods~\cite{mallat09:_wavel_tour_signal_proces}.
For two-dimensional images, wavelets have become part
of the JPEG2000 standard~\cite{christopoulos2000jpeg2000}, and they
have also been applied successfully for the compression of tomographic
volume data recently~\cite{sauerEtAl19:_compressed_computing,stockEtAlSauer20:_edge_ct}.

\begin{figure}
  \begin{center}
    \framebox{\includegraphics[width=\imagewidthincolumn]{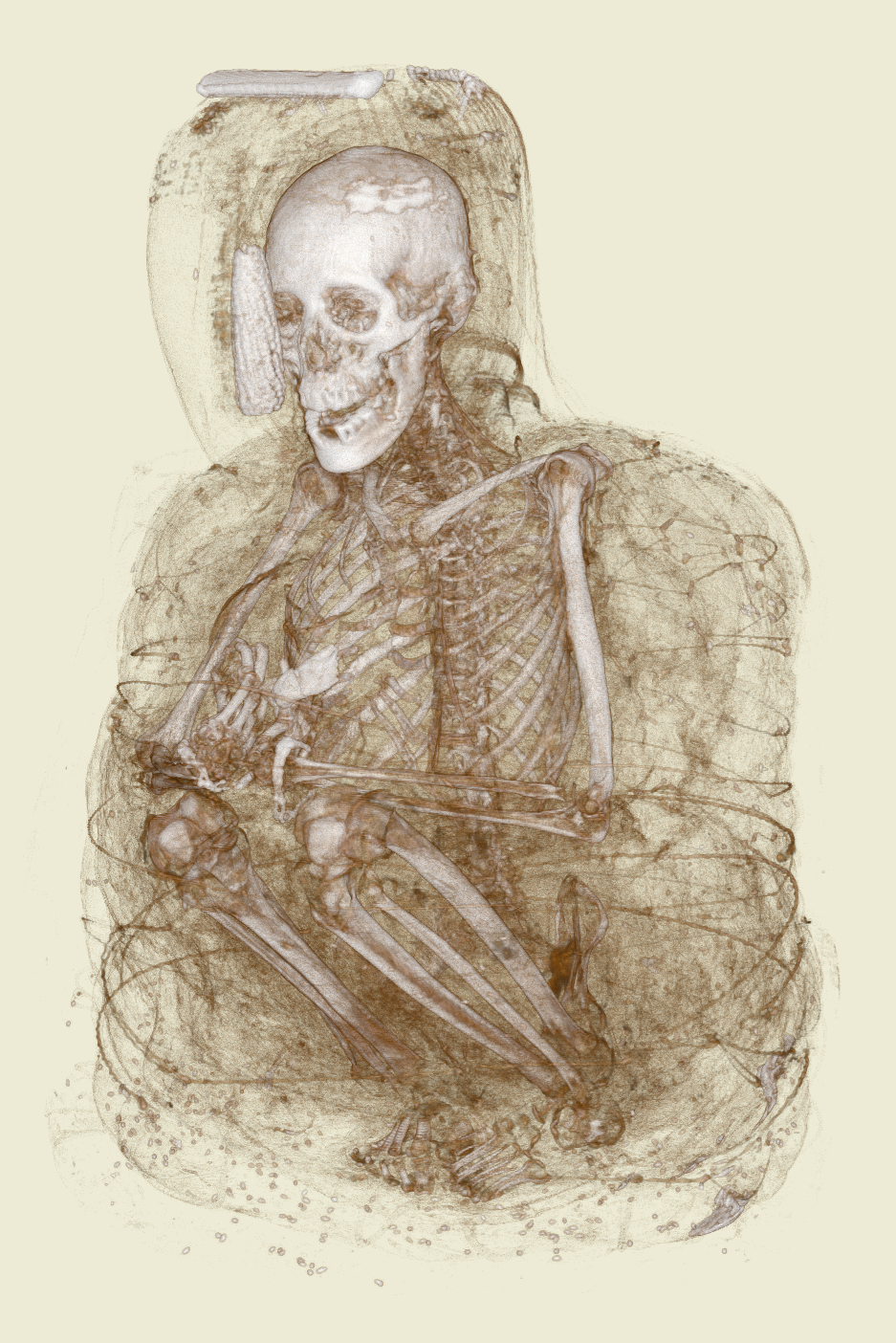}}
  \end{center}
    \caption{Overview rendering of the complete mummy dataset at a coarse resolution. Image: Fraunhofer~IIS/Christoph Heinzl\label{fig:mummyFull}}
\end{figure}
\begin{figure}
  \begin{center}
    \framebox{\includegraphics[width=\imagewidthincolumn]{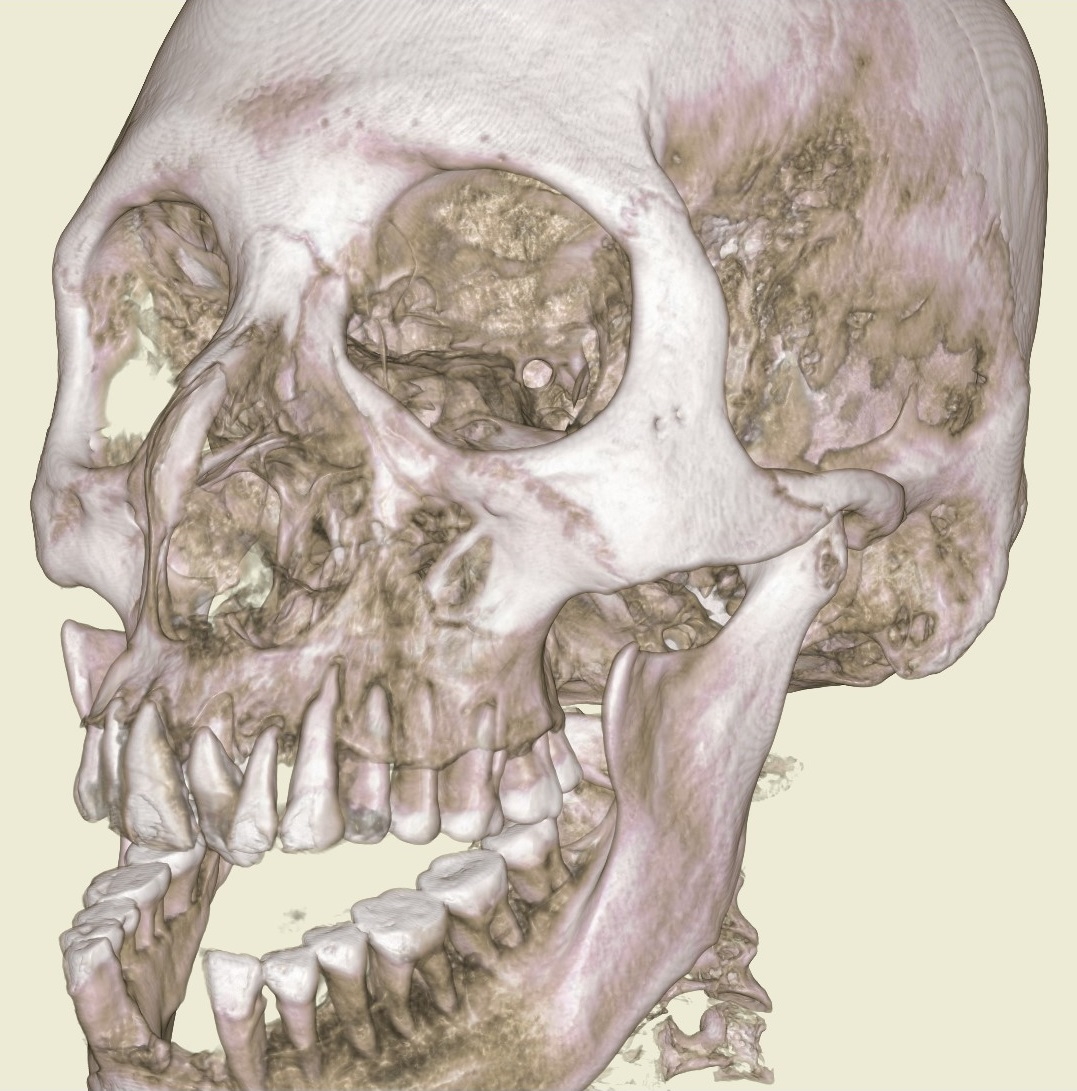}}
  \end{center}
    \caption{Rendering of full-resolution details of the mummy skull and teeth.
      All levels of detail are obtained from the locally decompressed wavelet
      dataset only. Image: Fraunhofer~IIS/Christoph Heinzl\label{fig:mummyTeeth}}
\end{figure}
To illustrate the type of objects that we are concerned with, we refer
to the CT scan of a Peruvian mummy, located at the Linden-Museum in Stuttgart,
that was performed at the Fraunhofer~IIS Development Center X-ray Technology
(EZRT), see Fig.~\ref{fig:mummyFull} and~\ref{fig:mummyTeeth}.
The size of the dataset is roughly \SI{970}{\giga\byte} and in order to handle
and visualize it on consumer hardware, it has been reduced to about
\SI{30}{\giga\byte} by 
the wavelet compression method from \cite{stockEtAlSauer20:_edge_ct}.
While high-resolution regions of interest (ROIs) of the scan can be obtained
on-the-fly, the full resolution image is not tractable with reasonable effort
and this necessitates algorithms that work exclusively on the wavelet coefficients.

The simplest choice for selecting a wavelet is the
\emph{Haar wavelet}~\cite{haar10:_zur_theor_funkt}
and its associated multiresolution analysis generated by
characteristic functions of the unit interval and their dyadic
refinement properties. Haar wavelets have the advantages of being
compactly supported, orthonormal and symmetric, and they are the only
wavelet system with these properties~\cite{Meyer93}. Moreover, Haar
wavelets have minimal support among all discrete multiresolution
systems and thus provide optimal localization. Fortunately, they can
also be implemented in a very efficient way, allowing for fast
decomposition and on-the-fly reconstruction of images. This by itself
makes them interesting and useful despite their well-known drawbacks
like a lack of smoothness and vanishing moments.

In this paper we show that and how Haar wavelets can be used for gradient
estimation and an approximate total variation (TV) denoising directly
on the wavelet coefficients without any need to reconstruct the full image.
This is a fundamental requirement for handling very large images, especially
in 3D. By this direct and computationally cheap manipulation of wavelet
coefficients, TV denoising can even be integrated into an almost real-time
visualization pipeline.

Moreover, the proofs will show that the approximation of the TV norm
is a particular property of
tensor product Haar wavelets: the fact that they have only \emph{one
vanishing moment} is responsible for their ability to approximate the
gradient on a grid. Their symmetry and analytic expression allow us
to determine the explicit level-dependent renormalization coefficients that
guarantee that direction and length of the gradients are met
accurately. Altogether, this leads to a novel shrinkage scheme, where
vectors consisting of a subset of the wavelet coefficients are
thresholded with respect to their length after a proper dimension-dependent
and level-dependent renormalization.

The paper is organized as follows: after briefly discussing some
related work in Section~\ref{sec:RelWork}, we recall
the definition of Haar wavelets in Section~\ref{sec:Haar}, set up the
notation for arbitrary dimensions and derive a saturation result for
the wavelet coefficients depending on their type, more precisely, on
the distribution of 
scaling and wavelet components in the tensor product
function. Section~\ref{sec:TVApprox} shows how the TV norm of a
function can be approximated by a properly renormalized subset of
wavelet coefficients and gives explicit error estimates for this
approximation. How TV denoising based on the ROF functional can be
done directly on the wavelet coefficients is demonstrated in
Section~\ref{sec:TVReg}. Finally, in Section~\ref{sec:Applications} we
give some example applications of the method on CT datasets.

\section{Related work}
\label{sec:RelWork}

It has long been observed that soft thresholding of (especially Haar)
wavelet
coefficients is a useful tool for data denoising. A careful
mathematical analysis and discussion of the relationship 
between Haar 
wavelet shrinkage and TV regularization has been given in~\cite{steidl02:_relat_soft_wavel_shrin_total_variat_denois}
and~\cite{welk08:_local}, starting from the study of diffusion
processes. More precisely, the authors showed that their diffusion-based
process of denoising has \emph{strong connections} to Haar wavelet
shrinkage. Though this approach is likely superior to ours for the goal of
denoising of two-dimensional images, it is not applicable to our needs as it
relies on an iterative diffusion process of the complete image that is
not practically feasible for gigavoxel datasets.

The geometric meaning of Haar wavelet coefficients has also been used 
implicitly in~\cite{reisenhofer18:_haar_wavel_based_percep_simil} to determine
similarity indices which are essentially based on the role of a subset
of the Haar wavelet coefficients as approximations of gradient
vectors. This indeed might be carried over to three dimensions and
could be directly applied for quality measurements between voxel datasets.

Bounded variation in the context of piecewise constant functions was
studied in depth in~\cite{cohen99:_nonlin_approx_space_r} where it was
also shown that approximate solutions of the TV denoising problem can
be obtained by suitable thresholding of a Haar wavelet
decomposition. While their results rely on significantly weaker
assumptions, they are developed
only for the bivariate case where the
dimension-dependent renormalization of the approximative gradients
with respect to the level does not appear, see Remark~\ref{remark:normalization}.

\section{Haar wavelets}
\label{sec:Haar}
To set up notation, we begin by recalling the well-known concept of tensor
product \emph{Haar wavelets} on~$\RR^s$, where~$s\in\NN$ stands for the number
of variables. We denote the characteristic function on an interval~$I\subset\RR$
by~$\chi_I$. In the one-dimensional case, the Haar wavelets are based on the
\emph{scaling function}~$\psi_0$ and the \emph{wavelet}~$\psi_1$, defined as
\begin{equation}
  \label{eq:HaarWavDef}
  \psi_0 := \chi_{[0,1]}, \qquad \psi_1 :=
  \chi_{[0,\frac12]} - \chi_{[\frac12,1]}, 
\end{equation}
cf.~\cite{mallat09:_wavel_tour_signal_proces}. Since
$\psi_{0,k} := \psi_0 (\cdot - k)$ and the normalized
wavelets $\psi_{1,k}^n := 2^{n/2} \psi_1 \left( 2^n \cdot - k
\right)$, $k \in \ZZ$, $n \in \NN_0$, form an orthonormal basis of
$L_2 (\RR)$, any function $f \in L_2 (\RR)$ can be expressed, for $k
\in \ZZ$ and  $n \in \NN_0$, in terms
of its scaling coefficients 
\begin{equation}
c_k (f) := \int_\RR f(t) \psi_0 (t-k) \, dt,
\end{equation}
and wavelet coefficients
\begin{equation}
d_k^n (f) := 2^{n/2} \int_\RR f(t) \psi_1 ( 2^n t-k) \, dt,
\end{equation}
via
\begin{equation}
  \label{eq:WaveletDecomp1d}
  f = \sum_{k \in \ZZ} c_k (f) \, \psi_{0,k} + \sum_{n=0}^\infty \sum_{k
    \in \ZZ} d_k^n (f) \, \psi_{1,k}^n,
\end{equation}
and the wavelet coefficients of different levels give rise to a
\emph{multiresolution analysis}.

The extension to~$s$ variables by means of tensorization is
straightforward. Let $\theta \in \{0,1\}^s$ be the index for a basis
function and define, for $x = (x_1,\dots,x_s) \in \RR^s$,
\begin{equation}
  \label{eq:TensorprodHaar}
  \psi_\theta (x) := \prod_{j=1}^s \psi_{\theta_j} (x_j), \qquad
  \theta \in \{0,1\}^s \setminus \{ 0 \}.
\end{equation}
We call~$\psi_0 = \chi_{[0,1]^s}$ the \emph{scaling function} while
all other functions~$\psi_\theta$, $\theta \in \{0,1\}^s \setminus \{
0 \}$, are \emph{wavelets}, yielding $2^s - 1$ wavelet functions at each
level. With the scaling coefficients
\begin{equation}
c_\alpha (f) = \int_{\RR^s} f(t) \psi_0 (t-\alpha) \, dt, \qquad
\alpha \in \ZZ^s,
\end{equation}
and the wavelet coefficients
\begin{equation}
d_{\theta,\alpha}^n (f) = 2^{ns/2} \int_{\RR^s} f(t) \psi_\theta (2^n
t-\alpha) \, dt
\end{equation}
with indices $\alpha \in \ZZ^s$ and $n \in \NN_0$, we obtain the orthogonal
representation
\begin{equation}
  \label{eq:WaveletDecomp}
  f = \sum_{\alpha \in \ZZ^s} c_\alpha (f) \, \psi_{0,\alpha} +
  \sum_{n=0}^\infty \sum_{\theta \neq 0} \sum_{\alpha \in \ZZ^s}
  d_{\theta,\alpha}^n (f) \, \psi_{\theta,\alpha}^n
\end{equation}
with the normalized wavelet functions
\begin{equation}
  \label{eq:WaveletDef}
  \psi_{\theta,\alpha}^n := 2^{ns/2} \psi_\theta (2^n \cdot - \alpha). 
\end{equation}
Of course, in applications one usually does not work with the infinite
series, but only with a finite sum of wavelet levels up to some
maximal level~$n_1$.

\subsection{Wavelet moments}
\label{sec:HaarMoments}

As a first auxiliary result, we compute moments of the wavelets
$\psi_\theta (\cdot - \alpha)$ with respect to the polynomials $(\cdot +
\frac12 \epsilon)^\gamma$, $|\gamma| \le |\theta|$, where ${\epsilon :=
(1,\dots,1) \in \ZZ^s}$ and where $|\theta|$ denotes the
length of the multiindex~$\theta$. 
The scaled and shifted wavelets have the center
of their support at the points
\begin{equation}
  \label{eq:xalphan}
  x_\alpha^n := 2^{-n} \left( \alpha + \frac12 \epsilon \right).
\end{equation}
Note that these points form a set of non-interlacing grids in~$\RR^s$.
Indeed, if there were $n < n'$ and $\alpha, \alpha' \in
\ZZ^s$ such that $x_\alpha^n = x_{\alpha'}^{n'}$, then this would mean
that
$
2^{n'-n} \left( \alpha + \frac12 \epsilon \right) = \alpha' + \frac12
\epsilon
$
which is impossible since the left hand side is a point in~$\ZZ^s$
while the one on the right hand side lies in the shifted grid~$\ZZ^s +
\frac12 \epsilon$. The moments of monomials centered at such
midpoints with respect to the Haar wavelets have a simple explicit
form which we record first.

\begin{lemma}
  \label{L:gammathetMoment}
  For $\theta \in \{0,1\}^s$ and $|\gamma| \le |\theta|$ one has
  \begin{eqnarray}
    \nonumber
    \lefteqn{\int_{\RR^s} \left( x - x_\alpha^n \right)^\gamma \,
    \psi_{\theta,\alpha}^n (x) \, dx} \\
    \label{eq:gammathetMoment}
    & = & (-1)^{|\theta|}
    2^{-(n+2)|\theta| - ns/2} \, \delta_{\gamma,\theta}.
  \end{eqnarray}
\end{lemma}

\begin{proof}
  The integral in~\eqref{eq:gammathetMoment} can be written as
  \begin{eqnarray}
    \nonumber
    \lefteqn{\int_{\RR^s} \left( x - x_\alpha^n \right)^\gamma \,
    \psi_{\theta,\alpha}^n (x) \, dx} \\
    \nonumber
    & = & \prod_{j=1}^s
          \int_{\RR} \left( x_j - 2^{-n} \left(\alpha_j + \frac12 \right)
          \right)^{\gamma_j} \, \\
    & & \times \psi_{\theta_j,\alpha_j}^n (x_j) \,
          dx_j.
    \label{eq:gammathetMomentPf1}
  \end{eqnarray}
  We first note that if~$|\gamma| \le |\theta|$ and~$\gamma \neq
  \theta$ then there exists $k \in \{1,\dots,s\}$ such that $\gamma_k
  = 0$ and $\theta_k = 1$ and therefore
  \begin{equation}
  \int_{\RR} \left( x_j - x_\alpha^n \right)^{0} \,
  \psi_{1,\alpha_j}^n (x_j) \, dx_j = 0.
  \end{equation}
  Hence, one factor in~\eqref{eq:gammathetMomentPf1} vanishes and
  therefore the whole product. In consequence, the integral
  in~\eqref{eq:gammathetMoment} is nonzero only for~$\gamma = \theta$. To
  compute the value of the integral, we note that by shift invariance
  we can restrict ourselves to the case~$\alpha=0$. In the univariate case
  we note that 
  \begin{equation}
  \int_\RR \left( x - 2^{-n-1}
  \right)^{0} \, \psi_{0,0}^n  (x) \, dx
  = 2^{-n/2},
  \end{equation}
  as well as
  \begin{equation}
  \int_\RR \left( x - 2^{-n-1}
  \right) \, \psi_{1,0}^n (x) \, dx
  = -2^{-\frac32 n-2},
  \end{equation}
  so that the integral in~\eqref{eq:gammathetMoment} takes on the
  value
  \begin{eqnarray}
  \lefteqn{\hspace*{-2em}\left(\prod_{\substack{j\in\{1,\ldots,s\},\\\theta_j = 0}} 2^{-n/2} \right) \left( \prod_{\substack{j\in\{1,\ldots,s\},\\\theta_j = 1}}
    \left(-2^{-\frac32 n-2}\right)\right)} \nonumber \\
    & = & (-1)^{|\theta|} 2^{-\frac{n}2 (s-|\theta|)} \, 2^{\left(
          -\frac32 n - 2\right) |\theta|} \nonumber \\
    & = & (-1)^{|\theta|} 2^{-(n+2)|\theta| - ns/2},
  \end{eqnarray}
  whenever $\gamma = \theta$.
\end{proof}

\subsection{Limits of coefficients}
\label{sec:CoeffLim}

The simple observation of Lemma~\ref{L:gammathetMoment} allows us to
determine the limits of Haar wavelet coefficients for sufficiently smooth
functions and their natural rate of decay.

\begin{theorem}
  \label{T:HaarCoeffLim}
  For $\theta \in \{0,1\}^s$ suppose that $f \in C^{|\theta|+1}
  \left(\RR^s\right)$. Then, for any $\alpha \in \ZZ^s$,
  \begin{equation}
    \label{eq:HaarCoeffLim}
    \lim_{n \to \infty} \left| \kappa_{n,|\theta|} \,
      d_{\theta,\alpha}^n (f) - \frac{D^\theta f \left(x_\alpha^n
        \right)}{\theta!} \right| = 0,
  \end{equation}
  where $\kappa_{n,|\theta|} := (-1)^{|\theta|} 2^{(n+2)|\theta| +
    ns/2}$.
\end{theorem}

\begin{proof}
  For~$n \in \NN_0$, $\alpha \in \ZZ^s$, we consider the
  $(|\theta|+1)$-st Taylor expansion of~$f$ at~$x_\alpha$, to find that
  for any $x \in x_\alpha + \left[ -\frac12,\frac12 \right]^s$,
  \begin{eqnarray}
    f(x)
    & = & \sum_{|\gamma| \le |\theta|} \frac{D^\gamma
          f(x_\alpha^n)}{\gamma!} \, \left( x-x_\alpha^n
          \right)^{\gamma} \nonumber \\
    & & +
          \sum_{|\gamma| = |\theta|+1} \frac{D^\gamma f(\xi)}{\gamma!} \left(
          x-x_\alpha^n \right)^\gamma,  
  \end{eqnarray}
  where $\xi = \xi(x) \in x_\alpha^{n} + {2^{-n}} \left[
    -\frac12,\frac12 \right]^s$. Now, 
  by Lemma~\ref{L:gammathetMoment},
  \begin{eqnarray}
    \lefteqn{d_{\theta,\alpha}^n (f)} \nonumber\\
    & = & \int_{\RR^s} f(x) \, \psi_{\theta,\alpha}^n (x) \, dx \nonumber\\
    & = & (-1)^{|\theta|} 2^{-(n+2)|\theta| - ns/2} \, \frac{D^\theta f
          (x_\alpha^n)}{\theta!} \nonumber\\
    & & + \sum_{|\gamma| = |\theta|+1}
          \int_{\RR^s} \frac{D^\gamma f \left( \xi (x) \right)}{\gamma!} \left(
          x-x_\alpha^n \right)^\gamma \, \nonumber\\
    & & \qquad\qquad\qquad \times \psi_{\theta,\alpha}^n (x) \, dx.
  \end{eqnarray}
  The estimate
  \begin{eqnarray}
    \lefteqn{\left| \int_{\RR^s} \frac{D^\gamma f \left( \xi (x)
    \right)}{\gamma!} \left(x-x_\alpha^n \right)^\gamma \,
    \psi_{\theta,\alpha}^n (x) \, dx \right|} \nonumber\\
    & \le & \max_{x \in x_\alpha^{n} + 2^{-n} \left[ -\frac12,\frac12 \right]^s}
            \left| \frac{D^\gamma f(x)}{\gamma!} \right| \nonumber\\
    & & \times \; 2^{n/2} 
            \int_{ x_\alpha^{n} + 2^{-n} \left[ -\frac12,\frac12 \right]^s} \left|
            x-x_\alpha \right|^\gamma \, dx \nonumber\\
    & = & 2^{-(n+1)(|\theta|+1)-ns/2} \nonumber\\
    & & \times \max_{x \in x_\alpha^{n} + 2^{-n}
          \left[ -\frac12,\frac12 \right]^s} \left| \frac{D^\gamma
          f(x)}{\gamma!} \right|
  \end{eqnarray}
  holds for any $\gamma$ with $|\gamma| = |\theta|+1$, and immediately
  yields that 
  \begin{eqnarray}
    \nonumber    
    \lefteqn{\left| \kappa_{n,|\theta|}
    d_{\theta,\alpha}^n (f) - \frac{D^\theta f \left( x_\alpha^n
    \right)}{\theta!} \right|} \\
    \nonumber
    & \le & {|\theta|+s \choose s-1} 2^{|\theta|-n-1} \\
    \label{eq:THaarCoeffLimPf1}
    & & \times \max_{|\gamma| =
      |\theta|+1}\ \max_{x \in x_\alpha^{n} + 2^{-n}
      \left[ -\frac12,\frac12 \right]^s} \left| \frac{D^\gamma
        f(x)}{\gamma!} \right|,\quad
  \end{eqnarray}
  whose right hand side indeed tends to zero as $\mathcal{O}(2^{-n})$ for $n
  \to \infty$.
\end{proof}

If the $(|\theta|+1)$-st derivative of~$f$ is globally bounded in the
sense that
\begin{equation}
\left\| D^{|\theta|+1} f \right\|_\infty := \sup_{x \in \RR^s} \max_{|\gamma| =
  |\theta|+1} \left| \frac{D^\gamma f (x)}{\gamma!} \right| < \infty,
\end{equation}
then~\eqref{eq:THaarCoeffLimPf1} holds independently of~$\alpha$, and
we get the following improvement of Theorem~\ref{T:HaarCoeffLim}.

\begin{corollary}
  \label{C:HaarCoeffLimGlob}
  If $f \in C^{s+1}\left(\RR^s\right)$ and $\| D^k f \|_\infty < \infty$,
  $k=1,\dots,s+1$, then, for any $\theta \in \{0,1\}^s$,
  \begin{equation}
    \label{eq:HaarCoeffLimGlob}
    \lim_{n \to \infty} \sup_{\alpha \in \ZZ^s} \left|
      \kappa_{n,|\theta|} \, d_{\theta,\alpha}^n 
      (f) - \frac{D^\theta f \left( x_\alpha^n \right)}{\theta!} \right|
    = 0.
  \end{equation}  
\end{corollary}

Theorem~\ref{T:HaarCoeffLim} and its uniform version,
Corollary~\ref{C:HaarCoeffLimGlob},
have some consequences. The first one is that
Haar wavelet coefficients show what is known as a \emph{saturation behavior}
in Approximation Theory~\cite{Lorentz66}, i.e., they cannot decay
faster than a given rate, namely $\left| \kappa_{n,|\theta|}
\right|^{-1}$, no matter 
how smooth the underlying (nonpolynomial) function~$f$ is. Moreover,
we record the following for later reference. 

\begin{remark}[Wavelet coefficient decay]
  \label{R:CoeffDecayTheta}
  The decay rate of the coefficients depends
  on~$|\theta|$, which means that for smooth functions the
  coefficients decay faster whenever more ``wavelet contribution'' is
  contained in the respective wavelet function~$\psi_\theta$. In particular,
  in wavelet compression as in JPEG2000 where a hard thresholding
  is applied to the normalized wavelet coefficients, the
  coefficients~$d_{\theta,\alpha}^n (f)$ with a large $|\theta|$ are
  more likely to be eliminated by the thresholding process. This
  phenomenon is frequently observed in wavelet compression for images. 
\end{remark}

\subsection{Gradients}
\label{sec:Gradients}
From the wavelet coefficients we can
form, for~$\alpha \in \ZZ^s$ and~$n\in\NN_0$, the renormalized vectors
\begin{eqnarray}
  \label{eq:dbalphaDef}
  \nonumber
  \hat\db_\alpha^n (f) & := & -2^{n(1-s/2) + 2} \db_\alpha^n(f) \\
    & \hphantom{:}= & 2^{-ns} \kappa_{n,1} \db_\alpha^n(f)
\end{eqnarray}
with
\begin{equation}
  \label{eq:dbalphaDef1}
  \db_\alpha^n(f) :=
  \begin{pmatrix}
    d^n_{\epsilon_j,\alpha} (f) : j=1,\dots,s
  \end{pmatrix}
  \in \RR^s.
\end{equation}
By Theorem~\ref{T:HaarCoeffLim}, the vector~$2^{ns}
\hat\db_\alpha^n (f)$ is an approximation for the 
\emph{gradient} $\nabla f (x_\alpha^n)$, ${\alpha \in \ZZ^s}$, ${n \in
\NN_0}$. This can be used to extract gradient information directly from
the wavelet coefficient vectors $\kappa_{n,1} \, 
\begin{pmatrix}
  d^n_{\epsilon_j,\alpha} (f) : j=1,\dots,s
\end{pmatrix}$.
As an example, we consider the subsampled gradient fields for
piecewise smooth functions, computed directly from the wavelet
coefficients. For the $1$-norm and the squared $2$-norm, these are shown in
Fig.~\ref{fig:GradF1} and~\ref{fig:GradF2}, respectively.
\begin{figure}
  \begin{center}
    \framebox{\includegraphics[width=\imagewidthincolumn]{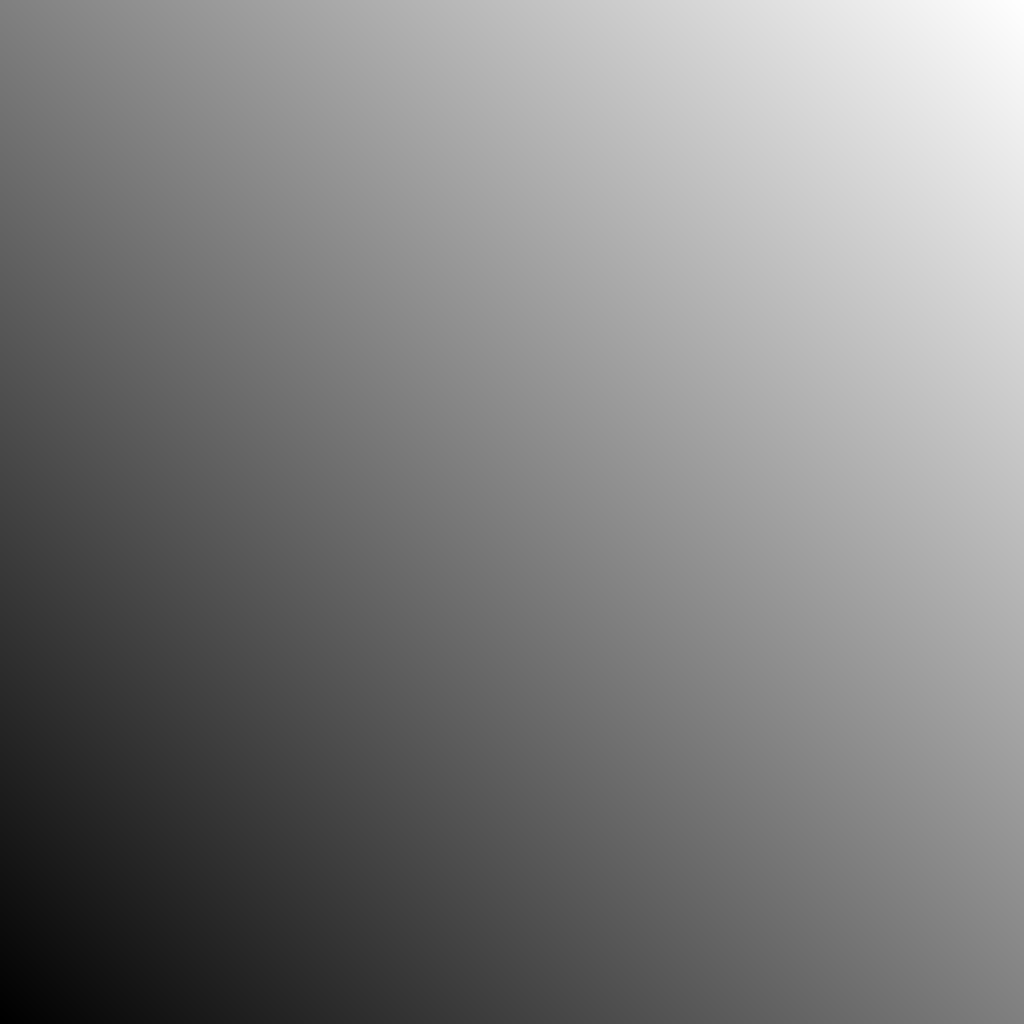}} \imagelinebreak \framebox{\includegraphics[width=\imagewidthincolumn]{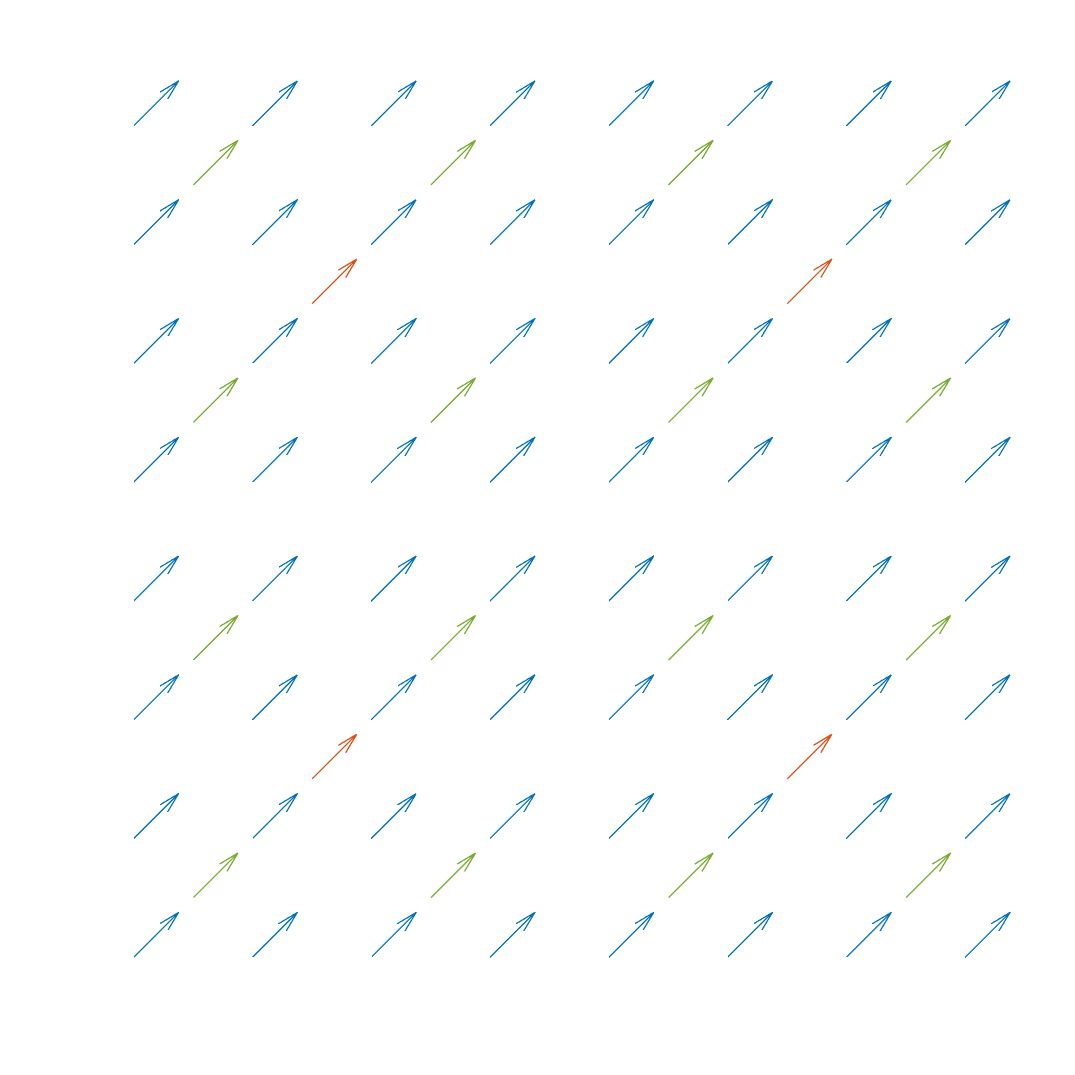}}
  \end{center}
  \caption{The level image \emph{(top)} and their vector fields of gradients for multiple levels~$n$ \emph{(bottom)} for ${f(x,y) = |x|+|y|}$ where~$x,y>0$\label{fig:GradF1}}
\end{figure}
\begin{figure}
  \begin{center}
    \framebox{\includegraphics[width=\imagewidthincolumn]{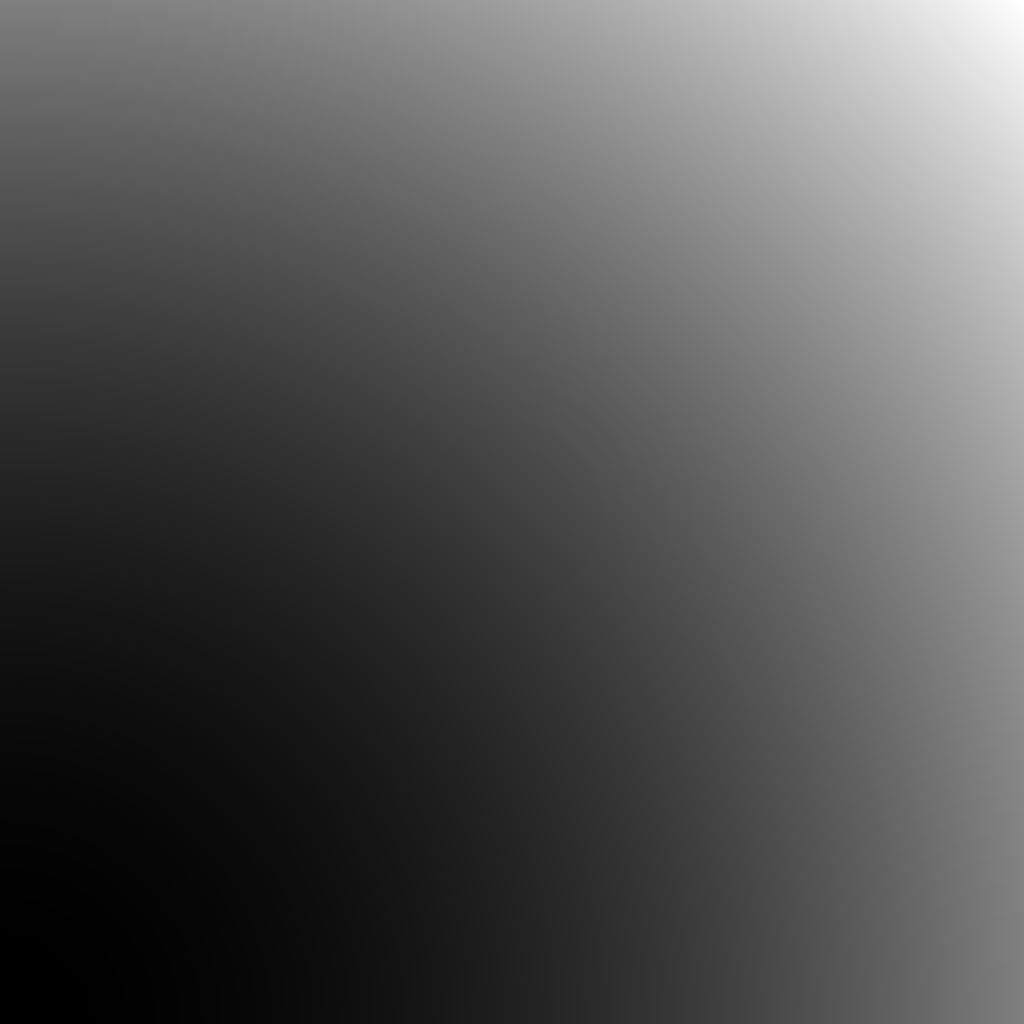}} \imagelinebreak \framebox{\includegraphics[width=\imagewidthincolumn]{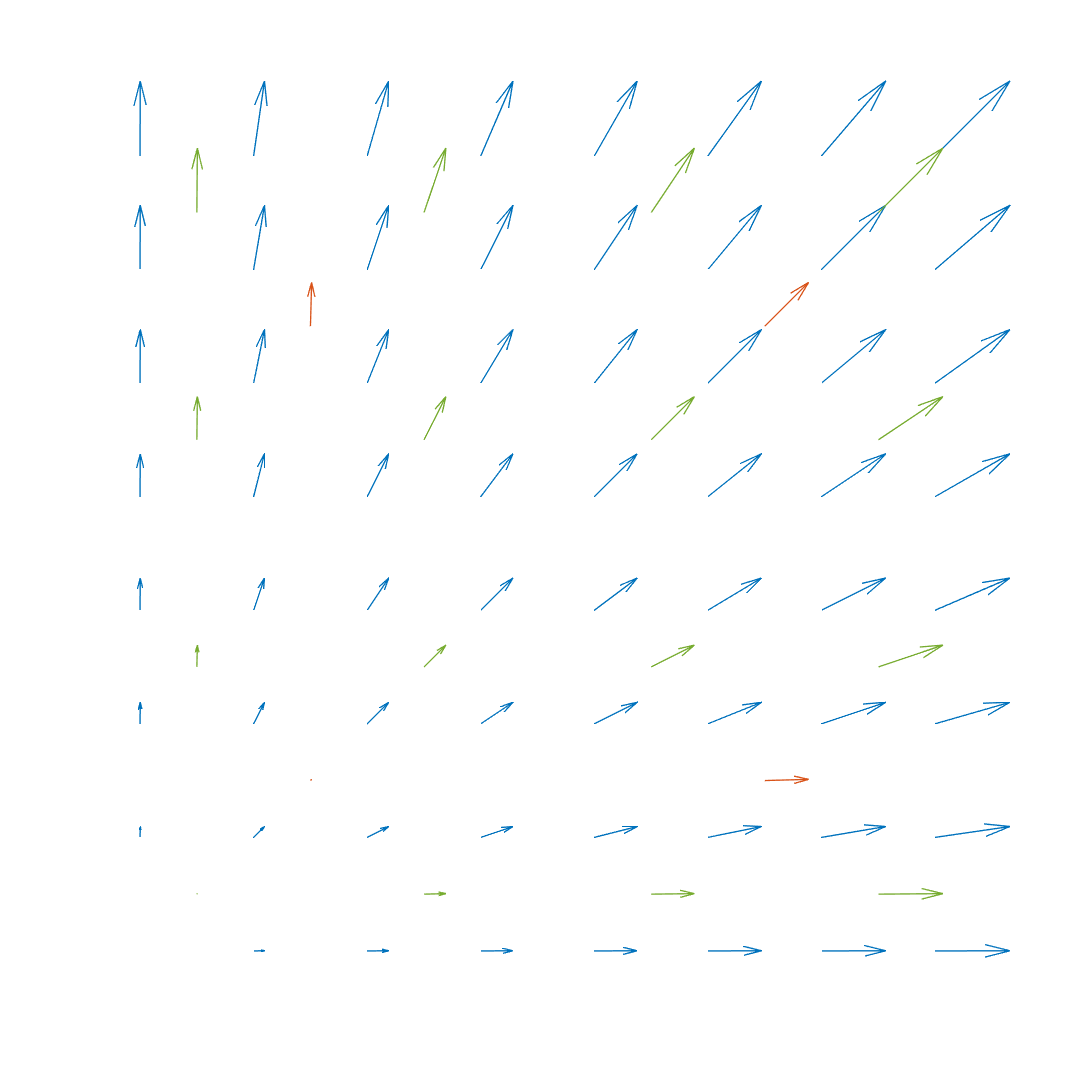}}
  \end{center}
  \caption{The level image \emph{(top)} and their vector fields of gradients for multiple levels~$n$ \emph{(bottom)} for ${f(x,y) = x^2 + y^2}$ where~$x,y>0$\label{fig:GradF2}}
\end{figure}
The situation changes in the case of non-smooth functions with sharp
contours. While in the binary image in Fig.~\ref{fig:GradF3} the \emph{directions} of 
the gradients are recognized correctly along the straight lines, their
lengths are upscaled by a factor of~$2^n$ which is due to the
different resolution levels as the difference between two neighboring
pixels is always either zero or one, but one divides by the
resolution-dependent distance between the pixels at different levels to obtain
the gradient. This can be compensated by rescaling the gradients by a
factor of~$2^n$ as shown on Fig.~\ref{fig:GradF3}. 

\begin{figure}
  \begin{center}
    \framebox{\includegraphics[width=\imagelesswidthincolumn]{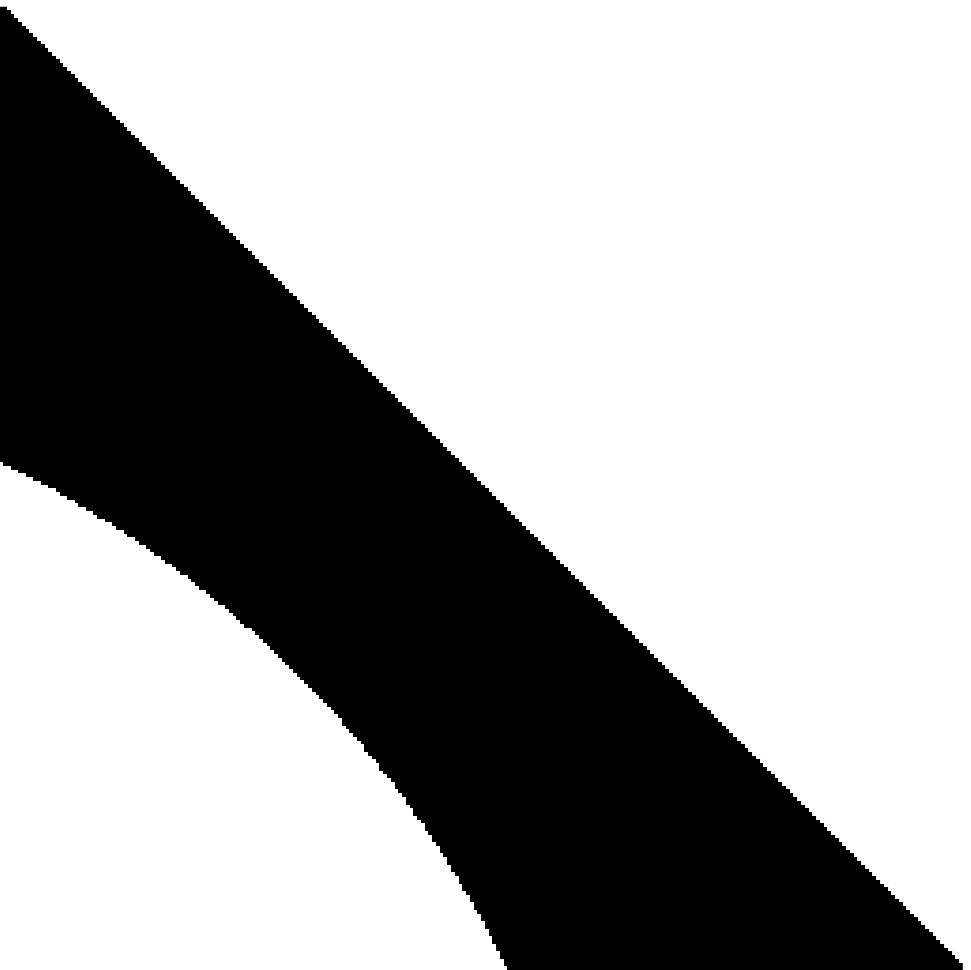}} \imagelinebreak
    \framebox{\includegraphics[width=\imagelesswidthincolumn,clip,trim={0 0.2cm 0 0}]{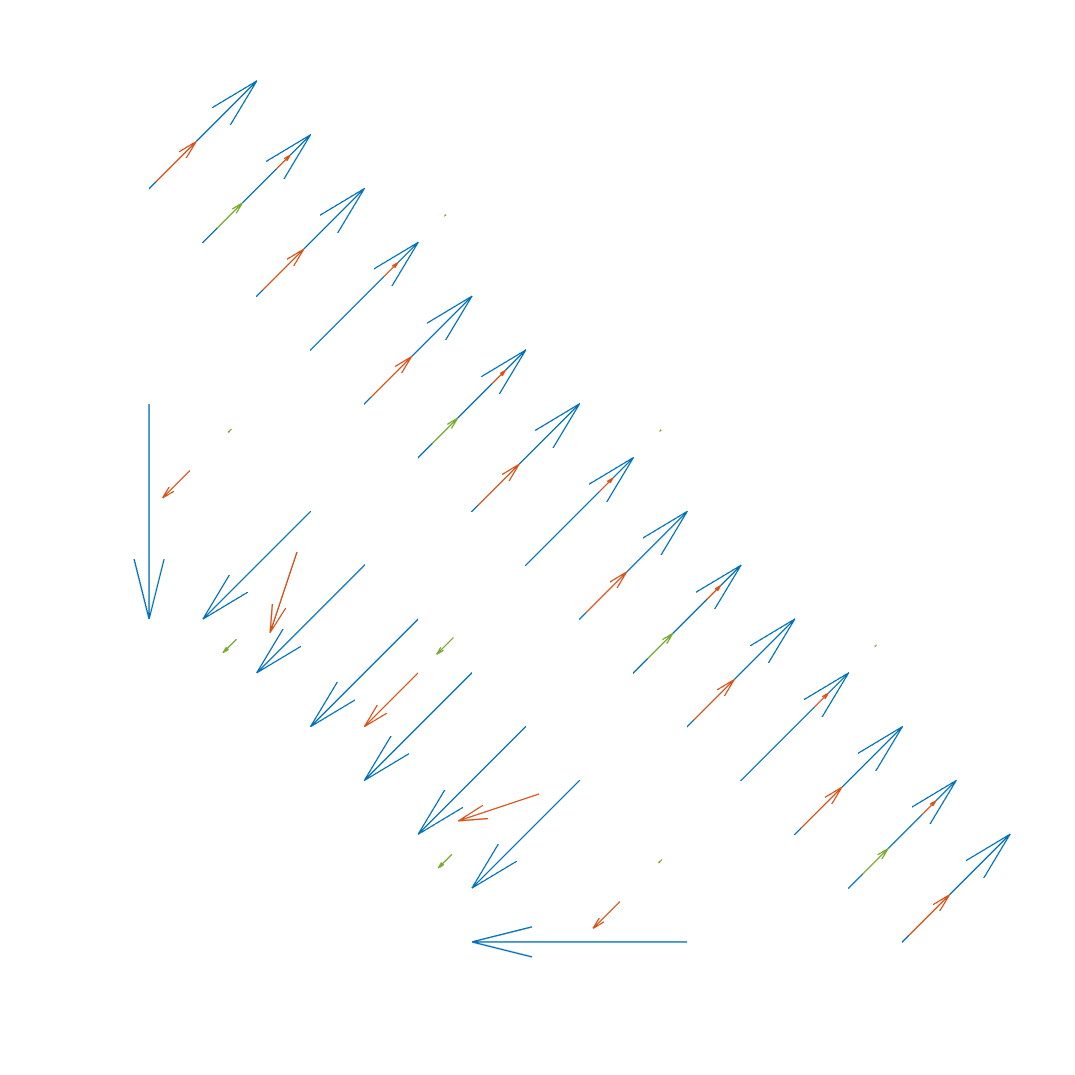}} \imagelinebreak
    \framebox{\includegraphics[width=\imagelesswidthincolumn,clip,trim={0 0.2cm 0 0}]{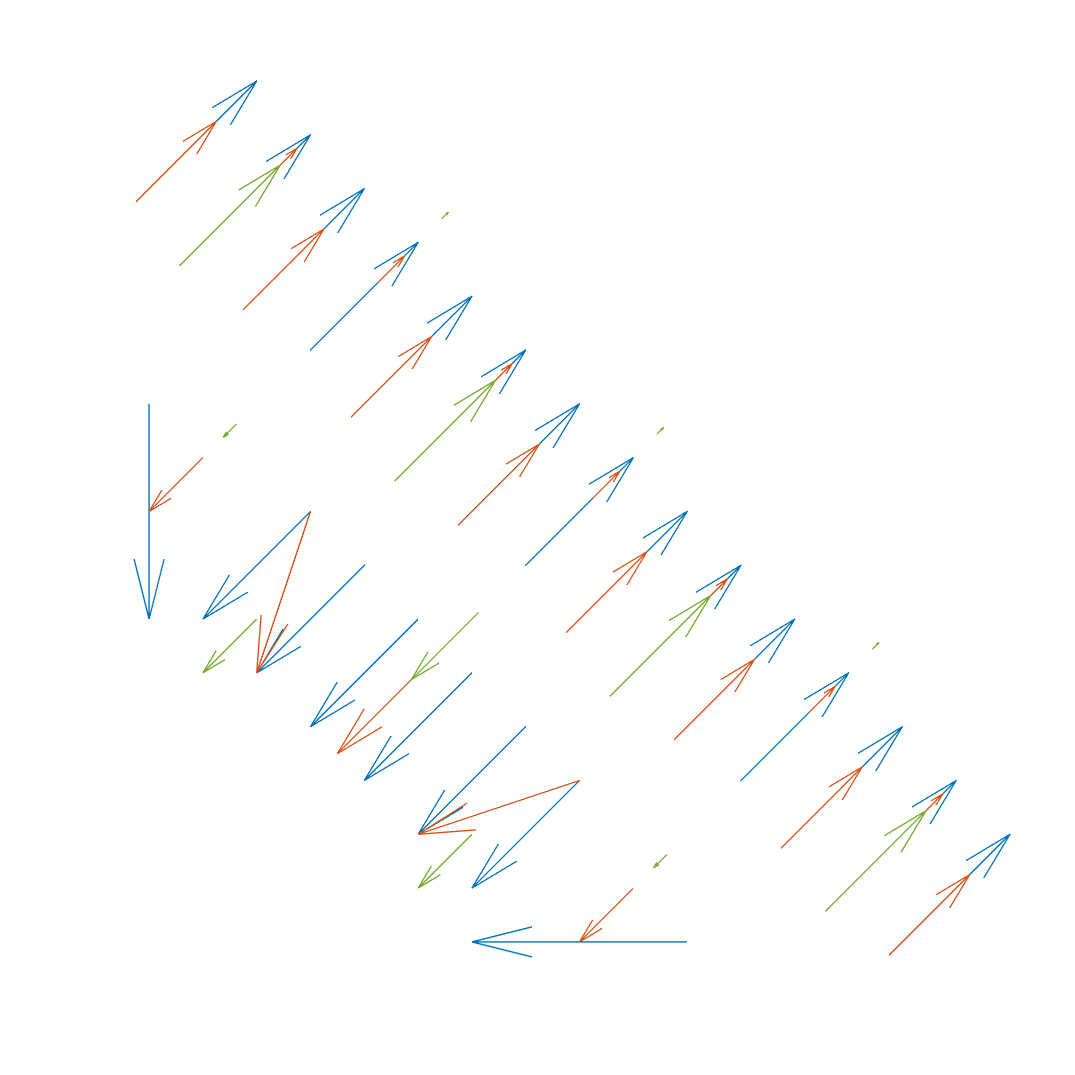}}
  \end{center}
  \caption{Simple synthetic image \emph{(top)} and the resulting gradient field (based on a low-resolution version for increased visual clarity),
    with the normalization factor~${2^{n(1+s/2) + 2} = 2^{2n+2}}$
    \emph{(middle)} and with a normalization~$2^{n+2}$ that compensates
    the resolution effects \emph{(bottom)}. In the images, the three
    colors stand for the three highest resolutions\label{fig:GradF3}}
\end{figure}

The estimate \eqref{eq:HaarCoeffLimGlob} is still a pointwise result,
even if the error is bounded uniformly with respect to~$\alpha$. We will 
explore this further in the next section to give an estimate of the
TV norm of~$f$ by means of wavelet coefficients, which also explains the
normalization chosen in~\eqref{eq:dbalphaDef}. So far, the estimates
in~\eqref{eq:HaarCoeffLim} and~\eqref{eq:HaarCoeffLimGlob} only work
in the supremum norm.

\subsection{Why Haar wavelets?}
\label{sec:WhyHaar}

The above derivation has been restricted to Haar wavelets for good
reasons. Besides the fact that their small support and simple
structure allows for fast computational methods like octrees, they
are the ones that can approximate gradients. In fact,
Lemma~\ref{L:gammathetMoment} can be easily extended to arbitrary
tensor product wavelets ${\psi (x) = \psi_1 (x_1) \otimes \cdots \otimes \psi_s
(x_s)}$ where $\psi_j$ has precisely $\nu_j$ vanishing moments. Then,
by the same arguments as in Lemma~\ref{L:gammathetMoment},
\begin{equation}
\int (x-x_\alpha^n)^\gamma \psi (x) \, dx = c_\gamma \,
\delta_{\gamma,\nu}
\end{equation}
whenever $|\gamma| \le |\nu|$ and the properly normalized wavelet
coefficients $d_{\theta,\alpha}^n (f)$ converge to $D^{\theta \cdot \nu}
f (x_\alpha)$, where $\theta \cdot \nu = \left( \theta_1
  \nu_1,\dots,\theta_s \nu_s \right)$ for sufficiently smooth~$f$.
Hence, an approximation of the gradient can be achieved if and
only if one uses tensor products of wavelets with only one vanishing
moments, i.e., $\nu = (1,\dots,1)$. Among these, Haar wavelets are the
ones of minimal support and we can compute the normalization
constants explicitly in an elementary way.

\section{Approximation of the TV norm}
\label{sec:TVApprox}

We will now derive estimates to show that the vectors
$\hat\db_\alpha^n (f)$ formed from the wavelet coefficients with only one
wavelet contribution, i.e., $d^n_{\epsilon_j,\alpha} (f)$ for $j=1,\dots,s$ in
Eq.~\eqref{eq:dbalphaDef1}, give a good approximation for the TV norm of~$f$
again directly from the wavelet coefficients.

\subsection{General remarks}
\label{sec:GenRem}

In the rest of this section, we make the assumption that the Haar
wavelet coefficients $d_{\theta,\alpha}^n (f)$ stem from a function~$f\in C^2\left(\RR^s\right)$
and the estimate will be given in terms of the magnitude
of its second derivatives. This is mainly for technical reasons and
will enable us to give quantitative estimates for the quality of
the approximate TV norm. This assumption includes, for example, the
case where $f$ is a bandlimited function, i.e., a function whose
Fourier transform~$\hat f$ is supported on some compact set $\Omega$,
a frequent assumption in computed tomography,
cf.~\cite{natterer86:_mathem_comput_tomog}. Moreover, it could be
extended to piecewise $C^2$ functions, another common assumption
in industrial CT, by locality of the
wavelet coefficients.

Despite of all that, we are well aware that this is a very restrictive
class of functions and that this assumption cannot be made in
general. However, the main goal of the paper is to derive the proper
normalization of the wavelet coefficients, which will depend on the
dimensionality of the data, and the proper thresholding process
induced by the approximate problem, which will be a thresholding with
respect to the length of certain vectors formed by the wavelet
coefficients. These two aspects are determined on the dense subset of
$C^2$ functions and if at all a wavelet thresholding is going to work
for approximate TV, this way is straightforward.

\subsection{The approximate TV norm}
\label{sec:TVNorm}

First, let us recall that the TV norm of~$f\in C^2\left(\RR^s\right)$, defined as
\begin{equation}
  \label{eq:TVNormDef}
  \| f \|_{TV} := \left\| | \nabla f |_2 \right\|_1
  = \int_{\RR^s} \left| \nabla f(x) \right|_2 \, dx,
\end{equation}
where $| \cdot |_2$ denotes the Euclidean norm in~$\RR^s$. This functional
plays a fundamental role in many imaging applications, at least since
the work of Rudin, Osher and Fatemi~\cite{rudin92:_nonlin}, which made
TV regularization a standard method in image processing, in particular for
image denoising. For a survey on TV based applications, we refer the reader
to~\cite{chambolle10:_introd_total_variat_image_analy}.

In the following we will need the Frobenius norm of the second
derivative of $f \in C^2\left(\RR^s\right)$, i.e., for $ x\in\RR^s$ we set
\begin{equation}
\left| D^2 f (x) \right|_F := \left( \sum_{j,k=1}^s \left(\frac{\partial^2 f
      (x)}{\partial x_j \partial x_{k}} \right)^2 \right)^{1/2}.
\end{equation}
We also define for $n \in \NN_0$, the functional
\begin{eqnarray}
  \label{eq:HnDef}
  \nonumber
  & & \hspace{-2em} H_n (f) \\
  & & \hspace{-2em} := \sum_{\alpha \in \ZZ^s} 2^{-ns} \max_{x \in x_\alpha^n +
    2^{-n} [ -\frac12,\frac12 ]^{s}} \left| D^2 f (x) \right|_F,
\end{eqnarray}
which may in general take on the value~$\infty$, but has the property
that $H_{n+1} (f) \le H_n (f)$, $n \in \NN_0$, and is clearly bounded
from below. Functions for which $H_n (f) < \infty$, $n \in \NN_0$, are
for  
example compactly supported functions in $C^2\left(\RR^s\right)$ as then the
Frobenius norm is globally bounded and the sum in~\eqref{eq:HnDef} becomes a
finite one. Since bandlimited functions decrease exponentially and
have a bandlimited second derivative, they
also satisfy $H_n (f) < \infty$, $n \in \NN_0$.

We will prove in this section that the
$\ell_1$ norm of the sequence $\left| \hat\db_\alpha^n(f) \right|_2$,
$\alpha\in\ZZ^s$, given as
\begin{eqnarray}
  \nonumber
  \lefteqn{\left\| \left| \hat\db^n(f) \right|_2 \right\|_1
  := \sum_{\alpha \in \ZZ^s} \left| \hat\db_\alpha^n(f) \right|_2} \nonumber\\
  \label{eq:NormalizedDiscGrads}
  & \hphantom{:}= & 2^{n(1-s/2) + 2} \sum_{\alpha \in \ZZ^s} \left( \sum_{j=1}^s \left(
        d_{\epsilon_j,\alpha}^n(f) \right)^2 \right)^{1/2}
\end{eqnarray}
with $\hat\db^n(f) := \left( \hat\db_\alpha^n(f) : \alpha \in \ZZ^s \right)$,
is an approximation to $\| f \|_{TV}$, and we provide an explicit
estimate for the error of this approximation. The result is as follows.

\begin{theorem}
  \label{T:WaveletTVApprox}
  If $f \in C^2\left(\RR^s\right)$ and $H_n (f) < \infty$ for some $n \in
  \NN_0$, then there exists, for any $n_0 \in \NN_0$, a constant $C >
  0$ that depends only on $n_0$ and $f$, such that
  \begin{equation}
    \label{eq:WaveletTVApprox}
    \left| \left\| \left| \hat\db^n(f) \right|_2 \right\|_1 - \| f \|_{TV}
    \right| \le C 2^{-n} \left\| \left| D^2 f \right|_F \right\|_1
  \end{equation}
  holds for all $n \ge n_0$.
\end{theorem}

The slightly strange formulation of Theorem~\ref{T:WaveletTVApprox}
will become clear from the proof. Indeed, the number~$n_0$ can be used
to control and reduce the constant~$C$ by relying on high-pass
information only. More details on that in Remark~\ref{R:n0Meaning}
later. Our main point in Theorem~\ref{T:WaveletTVApprox} is that
considering some part of the wavelet coefficients with a proper
dimension-dependent normalization gives a discrete approximation of
the TV norm of the underlying function with a quantitative error.

\begin{remark}[Normalization of coefficients]
\hphantom{.}
\label{remark:normalization}
  \begin{enumerate}
  \item The normalization of the coefficients $\hat\db_\alpha^n (f)$ in
    \eqref{eq:dbalphaDef} may be somewhat surprising at first view
    because of its dependency on the dimensionality~$s$. For $s=1$,
    coefficients at higher level are reweighted with an increasing
    weight $2^{n/2+2}$, for the image case $s=2$ the orthonormal normalization
    is just perfect, while for $s > 2$ the weights decrease and penalize
    higher levels, which is particularly relevant for our main
    application, namely the volume case $s=3$.
  \item Nevertheless, there is an explanation for this behavior. The TV
    norm considers the \emph{Euclidean length} of the gradient, hence a
    one-dimensional feature that scales like $2^{-n}$, while all other
    normalizations within the integrals are based on volumes that
    scale like $2^{-ns}$. This feature also appears when considering
    coefficients with respect to unnormalized wavelets.
  \item It is important to keep in mind that it is the proper
    \emph{renormalization} of the relevant coefficients and their
    treatment as vectors that allows us to approximate the TV norm by
    means of summing the coefficients.
  \end{enumerate}
\end{remark}

\subsection{Proof of Theorem~\ref{T:WaveletTVApprox}}
\label{sec:TechDet}

We split the proof of Theorem~\ref{T:WaveletTVApprox} into two parts,
first showing that $\left\| \left| \hat\db^n(f) \right|_2 \right\|_1$
is close to a cubature formula and then estimating the quality of the
cubature formula. To that end, we define the sequence
\begin{equation}
\fb_\nabla^n := \left( \nabla f \left(x_\alpha^n\right) : \alpha \in \ZZ^s
\right)
\end{equation}
of sampled gradients of~$f$.

\begin{lemma}
  \label{L:WaveletTVCubApprox}
  If $f \in C^2\left(\RR^s\right)$ and $H_n (f) < \infty$ for some $n \in
  \NN_0$, then there exists, for any $n_0 \in \NN_0$, a constant~$C >
  0$ that depends only on $n_0$ and $f$ such that
  \begin{equation}
    \label{eq:WaveletTVCubApprox}
    \left\| \left| \hat\db^n(f) \right|_2 - 2^{-ns} \left| \fb_\nabla^n
      \right|_2 \right\|_1 \le 2^{-n} C \left\| \left| D^2 f \right|_F
    \right\|_1
  \end{equation}
  holds for any $ n \ge n_0$.
\end{lemma}

\begin{proof}
  We again use a Taylor expansion at~$x_\alpha^n$, this time with an integral
  remainder. It follows directly by applying the univariate formula
  to~${t \mapsto f \left( x_\alpha + t(x-x_\alpha^n) \right)}$, $t \in
  [0,1]$, and takes the form
  \begin{eqnarray}
    \nonumber
    \lefteqn{f(x)}\\
    \nonumber
    & = & f(x_\alpha^n) + \nabla f(x_\alpha^n)^T
    (x-x_\alpha^n) \\
    \nonumber
    & & +
        \int_0^1 (1-\xi) (x-x_\alpha^n )^T 
        D^2 f \left( x_\alpha^n + \xi (x-x_\alpha^n) \right) \\
    \label{eq:Taylorsd}
    & & \qquad \times \; (x-x_\alpha^n) \, d\xi.
  \end{eqnarray}
  Using Lemma~\ref{L:gammathetMoment}, it follows for
  $j=1,\dots,s$ that
  \begin{eqnarray}
    \nonumber
    \lefteqn{d_{\epsilon_j,\alpha}^n (f)} \\
    \nonumber
    & = & \int_{\RR^s} f(x) \psi_{\epsilon_j,\alpha}^n (x) \, dx \\
    \label{eq:LWaveletTVCubApproxPf1}
    & = & -2^{-n(1+s/2)-2} \frac{\partial f}{\partial x_j} (x_\alpha^n) +
    h_{j,\alpha}^n(f),
  \end{eqnarray}
  where
  \begin{eqnarray}
    \lefteqn{h_{j,\alpha}^n(f)
    := \int_{\RR^s} \int_0^1 (1-\xi)} \nonumber\\
    & & \times \, (x-x_\alpha^n )^T D^2 f \left( x_\alpha^n + \xi
        (x-x_\alpha^n) \right) (x-x_\alpha^n) \nonumber\\
    & & \times \,
        \psi_{\epsilon_j,\alpha}^n (x) \, d\xi dx .
  \end{eqnarray}
  The magnitude of the latter can be estimated as
  \begin{eqnarray}
    \lefteqn{\left| h_{j,\alpha}^n(f) \right|} \nonumber\\
      & \le & 2^{ns/2} \int_0^1 (1-\xi) \nonumber\\
      & & \times \, \int_{2^{-n}
        [-\frac12,\frac12]^{s}} \left| x \right|_2^2
          \left| D^2 f \left(x_\alpha^n + \xi x \right)
          \right|_F \, dx \, d\xi \nonumber\\
      & \le & s 2^{-2n-2+ns/2} \int_0^1 (1-\xi) \nonumber\\
      && \times\int_{2^{-n}
         [-\frac12,\frac12]^{s}} 
         \left| D^2 f \left(x_\alpha^n + \xi x \right)
         \right|_F \, dx \, d\xi.
  \end{eqnarray}
  This estimate is independent of~$j$. Hence, if we define
  \begin{equation}
  \hb_\alpha^n(f) = \left( h_{j,\alpha}^n(f) : j=1,\dots,s \right),
  \end{equation}
  we get that
  \begin{eqnarray}
    \nonumber
    \lefteqn{\left| \hb_\alpha^n(f) \right|_2} \\
    \nonumber
    & \le & s^{3/2} 2^{-2n-2+ns/2} \int_0^1
    (1-\xi) \\
    \label{eq:LWaveletTVCubApproxPf1a}
    & & \times \, \int_{2^{-n} 
            [-\frac12,\frac12]^{s}} 
            \left| D^2 f \left(x_\alpha^n + \xi x \right)
            \right|_F \, dx \, d\xi.
  \end{eqnarray}
  The step functions
  \begin{eqnarray}
    g_+^n(f) & := & \sum_{\alpha \in \ZZ^s} \left( \max_{x \in x_\alpha^n
        + 2^{-n} [-\frac12,\frac12]^{s }} \left| D^2 f (x) \right|_{ F }
    \right) \nonumber\\
    & & \times\,
        \chi_{[0,1]^s} (2^n \cdot - \alpha)
  \end{eqnarray}
  converge monotonically decreasing and pointwise to $\left| D^2 f
  \right|_F$ as $n \to \infty$ and satisfy $0 \le \| g_+^n(f) \|_1 =
  H_n (f)$ which is finite 
  for sufficiently large $n$. Hence, using that $0\leq \left| D^2
    f\right|_F \leq g_+^n(f)$, 
  $\left| D^2 f \right| \in L_1 (\RR^s)$. In addition,
  \begin{eqnarray}
    g_-^n(f)
    & := & \sum_{\alpha \in \ZZ^s} \left( \min_{x \in x_\alpha^n
           + 2^{-n} [-\frac12,\frac12]^{ s }} \left| D^2 f (x)
           \right|_{ F } \right) \nonumber\\
    & & \times \, \chi_{[0,1]^s} (2^n \cdot - \alpha)
  \end{eqnarray}
  is trivially bounded from below by $0$, from above by $\left| D^2
    f \right|_F $ and converges monotonically increasing to
  $\left| D^2 f \right|_F$. Since
  for any $\xi \in [0,1]$ we have that
  \begin{eqnarray}
    \lefteqn{\min_{x \in x_\alpha^n + 2^{-n} [-\frac12,\frac12]^{ s }} \left| D^2 f (x)
    \right|_{ F }} \nonumber\\
    & \le &  2^{ns}  \int_{2^{-n} [-\frac12,\frac12]^{ s }} \left| D^2 f
            \left(x_\alpha^n + \xi x \right) \right|_F \, dx \nonumber\\
    & \le & \max_{x \in x_\alpha^n + 2^{-n} [-\frac12,\frac12]^{ s }} \left| D^2 f (x)
            \right|_F,    
  \end{eqnarray}
  and since $\| g_{\pm}^{n}(f) \|_1 \to \left\| \left| D^2 f \right|_F
  \right\|_1$ as $n\rightarrow\infty$, 
  it follows that
  \begin{align}
    \lim_{n \to \infty} & \sum_{\alpha \in \ZZ^s} \int_{2^{-n}
    [-\frac12,\frac12]^{ s }} 
    \left| D^2 f \left(x_\alpha^n + \xi x \right)
    \right|_F \, dx \nonumber\\
    & = \int_{\RR^s}  \left| D^2 f (x) \right|_F \, dx
  \end{align}
  uniformly in $\xi$. In particular, there exists, for any $n_0 \in
  \NN_0$, a constant $C = C(n_0,f)$ such that
  \begin{align}
  \sum_{\alpha \in \ZZ^s} & \int_{2^{-n}
    [-\frac12,\frac12]^{ s }} 
    \left| D^2 f \left(x_\alpha^n + \xi x \right)
                            \right|_F \, dx \nonumber\\
    & \le C \int_{\RR^s}  \left| D^2 f (x) \right|_F \, dx
  \end{align}
  holds for $n \ge n_0$. Moreover,
  \begin{equation}
    \label{eq:n0Limit}
    \lim_{n_0 \to \infty} C (n_0,f) = 1.
  \end{equation}
  Under the assumption that $n \ge n_0$ this yields that
  \begin{equation}
    \label{eq:LWaveletTVCubApproxPf2}
    \sum_{\alpha \in \ZZ^s} \left| \hb_{\alpha}^n(f) \right|_2 \le
    \frac{C}{2} s^{3/2} 2^{-2n-2+ns/2} \left\| \left| D^2 f \right|_F \right\|_1.
  \end{equation}
  Multiplying \eqref{eq:LWaveletTVCubApproxPf1} by $\kappa_{n,1} =
  -2^{n(1-s/2)+2}$ we 
  get that
  \begin{eqnarray}
  \lefteqn{\left| \kappa_{n,1} \, d_{\epsilon_j,\alpha}^n (f) - 2^{-ns}
    \frac{\partial f}{\partial x_j} (x_\alpha^n) \right|} \nonumber\\
  & & \le \left| \kappa_{n,1} \right| \, \left| h_{j,\alpha}(f) \right|,
  \end{eqnarray}
  hence,
  \begin{eqnarray}
    \lefteqn{\left| \left| \hat\db_\alpha^n(f) \right|_2 - 2^{-ns} \left| \nabla
    f(x_\alpha^n) \right|_2 \right|} \nonumber\\
    & \le & \left| \hat\db_\alpha^n(f) -
            2^{-ns} \nabla f(x_\alpha^n) \right|_2  \nonumber\\
    & \le & \left| \kappa_{n,1} \right| \,\left|
    \hb_\alpha^n(f) \right|_2.
  \end{eqnarray}
  Summing this over $\alpha \in \ZZ^s$ and substituting
  \eqref{eq:LWaveletTVCubApproxPf2} we then obtain
  \begin{align}
    \nonumber
    \sum_{\alpha \in \ZZ^s} & \left| \left| \hat\db_\alpha^n(f) \right|_2 -
                              2^{-ns} \left| \nabla f(x_\alpha^n)
                              \right|_2 \right| \\ 
    \label{eq:LWaveletTVCubApproxPf3}
    & \le 2^{-n}
    C \frac{s^{3/2}}{2} \left\| \left| D^2 f \right|_F \right\|_1,
  \end{align}
  which is \eqref{eq:WaveletTVCubApprox}.
\end{proof}

\begin{remark}[Choice of~$C$ and~$n_0$]
  \label{R:n0Meaning}
  The constant~$C$ in~\eqref{eq:WaveletTVCubApprox} can be chosen as
  $\frac{s^{3/2}}{2} + \varepsilon$ for any $\varepsilon > 0$ by
  making $n_0$ sufficiently large. In practice this means to avoid
  the low-pass content of the wavelet transformation which
  approximates the gradient only in a rather poor way. Of course, the
  dependency of the constant~$C$ on $f$ would also have to be
  taken into account. 
\end{remark}

Since \eqref{eq:WaveletTVCubApprox} immediately implies for
$n \ge n_0$ that
\begin{align}
  \nonumber
  & \left| \left\| \, \left| \hat\db^n(f) \right|_2 \, \right\|_1
  - \sum_{\alpha \in \ZZ^s} 2^{-ns} \left| \nabla f (x_\alpha^n)
    \right|_2 \right| \\
    \label{eq:dbfnabApprox}
  & \qquad \le C 2^{-n} \left\| \left| D^2 f \right|_F \right\|_1,
\end{align}
the discrete sum $\left\| \left| \hat\db^n(f) \right|_2 \right\|_1$
approximates the uniform \emph{cubature formula} for the TV norm. To
complete the proof, we only
have to recall the approximation quality of the cubature formula. This
follows by standard arguments which we include for completeness.

\begin{lemma}
  \label{L:CubatureEst}
  If $f \in C^2\left(\RR^s\right)$ and $H_n (f) < \infty$ for some $n \in
  \NN_0$, then there exists, for any $n_0 \in \NN_0$, a constant~$C >
  0$ that depends only on $n_0$ and $f$ such that
  \begin{eqnarray}
    \nonumber
    \lefteqn{\left| \left\| 2^{-ns} \left| \fb_\nabla^n \right|_2 \right\|_1 -
      \| f \|_{TV} \right|} \\
    \label{eq:CubatureEst}
    & \le & C 2^{-n} \left\| \left| D^2 f \right|_F \right\|_1, 
  \end{eqnarray}
  for $ n \ge n_0$.
\end{lemma}

\begin{proof}
  Writing
  \begin{eqnarray}
    \lefteqn{\nabla f (x)} \nonumber\\
    & = & \nabla f (x_\alpha^n) + \int_0^1 D^2 f \left(x_\alpha^n + \xi (x-x_\alpha^n) \right) \nonumber\\
    & & \times (x-x_n^\alpha) \, d\xi,
  \end{eqnarray}
  we get by similar transformations as in the proof of
  Lemma~\ref{L:WaveletTVCubApprox} that
  \begin{eqnarray}
  \lefteqn{\int\limits_{x_\alpha + 2^{-n} [-\frac12,\frac12]^{ s }}
    \left| \nabla f (x) - \nabla f(x_\alpha^n) \right|_2 \, dx} \nonumber\\ 
    & \le & \int\limits_{2^{-n}
    [-\frac12,\frac12]^{ s }} \int_0^1 \left| D^2 f \left(
      x_\alpha^n + \xi x \right) \right|_F  | x |_2 \, d\xi, \nonumber\\
  \end{eqnarray}
  hence
  \begin{eqnarray}
    \label{eq:LCubatureEstPf1}
    \lefteqn{\left| 2^{-ns} \left| \nabla f(x_\alpha^n) \right|_2 -
      \int\limits_{x_\alpha + 2^{-n} [-\frac12,\frac12]^{ s }} \left| \nabla f (x)
    \right|_2 \right|} \nonumber\\
    & \le & \sqrt{s} 2^{-n-1} \int\limits_{2^{-n}
      [-\frac12,\frac12]^{ s }} \int_0^1 \left| D^2 f \left(x_\alpha^n + \xi
        x \right) \right|_F \, d\xi. \nonumber\\
  \end{eqnarray}
  By the same arguments as in Lemma~\ref{L:WaveletTVCubApprox}, there
  exist $n_0\in\NN_0$ and $C >0$ such that
  \begin{align}
    & \left| \sum_{\alpha \in \ZZ^s} 2^{-ns} \left| \nabla f(x_\alpha^n)
      \right|_2 - \int_{\RR^s} \left| \nabla f(x) \right|_2 \, dx
      \right| \nonumber\\
    & \qquad \le \frac{\sqrt{s}}{2} 2^{-n} C \left\| \left| D^2 f
        \right|_F \right\|_1,
  \end{align}
  giving \eqref{eq:CubatureEst}.
\end{proof}

Now it is easy to complete the proof of the main theorem.

\begin{proof}[Proof of Theorem~\ref{T:WaveletTVApprox}]
  Combining \eqref{eq:dbfnabApprox} and \eqref{eq:CubatureEst}, the
  triangle inequality gives, for $n \ge n_0$,
  \begin{equation}
  \left| \left\| \left| \hat\db^n(f) \right|_2 \right\|_1 - \| f
    \|_{TV} \right| \le C 2^{-n} \left\| \left| D^2 f \right|_F
  \right\|_1,
  \end{equation}
  where the constant can be anything of the form
  \begin{equation}
  \frac{s^{3/2}}{2} + \frac{s^{1/2}}{2} + \varepsilon = \sqrt{s}
  \frac{s+1}{2} + \varepsilon, \qquad \varepsilon > 0,
  \end{equation}
  by selecting a sufficiently high value for $n_0$.
\end{proof}

\subsection{Estimation over several levels}
\label{sec:SevLev}

The estimate in~\eqref{eq:WaveletTVApprox} holds for any sufficiently
high wavelet level separately and the error of the estimate decreases
with $n$, so it might appear reasonable to approximate the TV norm
just by the maximal level. Unfortunately, wavelet coefficients of high
level are most affected by high frequency noise which would lead to an
overdetection of gradients. Therefore, it makes sense to incorporate
also wavelet coefficients of lower resolution.

Indeed, in practical applications one
starts with finite data on the finest resolution~$n_1+1$, i.e.,
\begin{equation}
c_\alpha^{n_1+1} (f) := \int_{\RR^s} f(x) \, \psi_{0,\alpha}^{n_1+1} (x)
\, dx, \qquad \alpha \in \ZZ^s,
\end{equation}
and then computes the coefficients in the wavelet decomposition
\begin{align}
  \sum_{\alpha \in \ZZ^s}
  & c_\alpha^{n_1+1} (f) \,
    \psi_{0,\alpha}^{n_1+1} \nonumber\\
  & = \sum_{k=0}^{n_1} \sum_{\theta \in \{0,1\}^s
    \setminus \{0\}} \sum_{\alpha \in \ZZ^s} d_{\theta,\alpha}^{k} (f) \,
    \psi_{\theta,\alpha}^{k}.
\end{align}
To get an approximation of the TV norm that uses as many levels as possible
at the same time, namely $\hat\db^n(f)$, $n = n_0,\dots,n_1$, we make
use of suitable averaging to obtain almost the same rate of
accuracy as by the highest level alone.

\begin{proposition}
  \label{P:AverageTV}
  If, under the assumptions of Theorem~\ref{T:WaveletTVApprox}, we
  define
  \begin{equation}
    \label{eq:AverageTVdb}
    \mu_n := \frac{2^{n-n_1}}{2-2^{n_0-n_1}}, \quad n = n_0,\dots,n_1,    
  \end{equation}
  then
  \begin{eqnarray}
    \nonumber
    \lefteqn{\left| {\sum_{n=n_0}^{n_1} \mu_n \, \left\| \left|\hat\db^n(f)
    \right|_2 \right\|_1} - \| f \|_{TV} \right|}\\
    \label{eq:AverageTV1Est}
    & \le & C
    {(n_1+1-n_0)} 2^{-n_1} \left\| \left| D^2 f \right|_F \right\|_1.
  \end{eqnarray}
\end{proposition}

\begin{proof}
  By construction,
\begin{equation}
\sum_{n=n_0}^{n_1} \mu_n = 1,
\end{equation}
and we thus have that
  \begin{eqnarray}
    \lefteqn{ \left|  \sum_{n=n_0}^{n_1} \mu_n \left\|  \left| \hat\db^n(f)
          \right|_2 \right\|_1 - \| f
    \|_{TV} \right|} \nonumber\\
    & = & \left| \sum_{\alpha\in\ZZ^s} \sum_{n=n_0}^{n_1} \mu_n \left|
          \hat\db^n_\alpha(f) \right|_2 - \sum_{n=n_0}^{n_1} \mu_n \|
          f \|_{TV} \right| \nonumber\\
    & \le & \sum_{n=n_0}^{n_1} \mu_n \left| \left\|
          \left| \hat\db^{ n }(f) \right|_2 \right\|_1 - \| f \|_{TV} \right| \nonumber\\
    & \le & \frac{C}{2 - 2^{n_0-n_1}} \left\| \left| D^2 f \right|_F
            \right\|_1 \sum_{n=n_0}^{n_1} 2^{n-n_1} 2^{-n} \nonumber\\
    & \le & (n_1+1-n_0) C 2^{-n_1} \left\| \left| D^2 f \right|_F
            \right\|_1,
  \end{eqnarray}
  which is~\eqref{eq:AverageTV1Est} with $C$ being the constant in
  Theorem~\ref{T:WaveletTVApprox} for initial
  level $n_0$.
\end{proof}

In fact, \emph{any} averaging of ${\left|\hat\db^n(f)\right|_2}$,
${n=n_0,\ldots,n_1}$, 
would yield an approximation for the TV norm, but the
particular choice of the weights in~\eqref{eq:AverageTVdb} ensures
that the rate of convergence obtained by this averaging process is the same
as that on the highest level $n_1$, 
only affected by the ``logarithmic'' number~$n_1-n_0$ of the levels
incorporated in the approximation process.

\section{Approximate TV regularization}
\label{sec:TVReg}

\emph{TV regularization} is a standard procedure for many imaging
applications nowadays, especially for denoising. It consists of solving,
for a given image~$f$,
an optimization problem of the basic form
\begin{equation}
  \label{eq:TVRegProb}
  \min_u \frac12 \left\| f - u \right\|_2^2 + \lambda \| u \|_{TV},
  \qquad \lambda > 0,
\end{equation}
where the regularization term $\lambda \| u \|_{TV}$ encourages a smooth
or less noisy behavior of~$u$ whose influence is controlled by the
parameter~$\lambda$. In most applications, $\| u \|_{TV}$ is
computed for discrete data~$u$ by numerical differentiation, usually by
means of differences. In particular, this not only requires access to
the full image, but the discrete gradient needs an additional amount of
$s$ times the memory consumption of the original image. Since this is
unacceptable in realistic applications, where the image is larger than
the available memory, the straightforward approach is to use the wavelet
coefficients as a computationally efficient approximation for $\| u \|_{TV}$.

A relaxation of the optimization problem~\eqref{eq:TVRegProb} can be solved
explicitly by standard methods that we are going to explain now. To
that end, we assume that $f$ and $u$ are given as 
finite orthonormal wavelet expansions
\begin{eqnarray}
  f & = & \sum_{\alpha \in \ZZ^s} c_\alpha (f) \, \psi_{0,\alpha} \nonumber\\
  & & +
      \sum_{n=0}^{n_1} \sum_{\theta \neq 0}
      \sum_{\alpha \in \ZZ^s} d_{\theta,\alpha}^n (f) \,
      \psi_{\theta,\alpha}^n, \nonumber\\
  u & = & \sum_{\alpha \in \ZZ^s} c_\alpha (u) \, \psi_{0,\alpha} \nonumber\\
  & & +
          \sum_{n=0}^{n_1} \sum_{\theta \neq 0}
          \sum_{\alpha \in \ZZ^s} d_{\theta,\alpha}^n (u) \,
          \psi_{\theta,\alpha}^n.
\end{eqnarray}
The first term in~\eqref{eq:TVRegProb} can be differentiated with
respect to the wavelet coefficients yielding
\begin{eqnarray}
  \frac{\partial}{\partial c_\alpha (u)} \frac12 \| f - u \|_2^2
  & = & {c_\alpha (u) - c_\alpha (f)}, \nonumber\\
  \frac{\partial}{\partial  d_{\alpha,\theta}^n (u)} \frac12 \| f - u \|_2^2
  & = & {d_{\alpha,\theta}^n (u) - d_{\alpha,\theta}^n (f)}.\quad
\end{eqnarray}
For the regularization term we now use the approximation from
Proposition~\ref{P:AverageTV} for $\| u \|_{TV}$, i.e., we solve the
approximate problem 
\begin{equation}
  \label{eq:TVRegAlternative}
  \min_{u} \frac12 \| f - u \|_2^2 + \lambda \sum_{n=n_0}^{n_1} \mu_n
  \, \left\| \hat\db^n (u) \right\|_1
\end{equation}
with the regularization term
\begin{equation}
F(u) =
\sum_{\alpha \in \ZZ^s} \sum_{n=n_0}^{n_1} \mu_n \left| \hat\db_\alpha^n (u) \right|_2,  
\end{equation}
whose (generally set-valued) subgradient is composed, for $\alpha \in
\ZZ^s$, and $n=n_0,\dots,n_1$, of
\begin{equation}
\partial_{\hat\db_\alpha^n (u)} F(u) = \mu_n
\begin{cases}
  B_1 (0), & \hat\db_\alpha^n(u) = 0, \\
  \left\{\dfrac{\hat\db_\alpha^n(u)}{\left| \hat\db_\alpha^n(u) \right|_2}\right\}, & \hat\db_\alpha^n(u)
  \neq 0, \\ 
\end{cases}
\end{equation}
where $B_1 (0) := \{ x \in \RR^s : |x|_2 \le 1 \}$ denotes the unit ball.
A necessary and sufficient condition for $u$ to be a solution of the
convex optimization problem~\eqref{eq:TVRegAlternative} is that
\begin{equation}
0 \in \partial \left( \frac12 \| f - \cdot \|_2^2 + \lambda \left\| {\hat\db (\cdot)
    } \right\|_1 \right) (u),
\end{equation}
cf.~\cite{Rockafellar70}, 
which is equivalent to setting all coefficients in the wavelet
expansion of $u$ equal to those of $f$, except the nontrivial conditions
\begin{equation}
  \label{eq:ApproxOptProbSol2}
  0 \in \left( \hat\db_\alpha^n (f) - \hat\db_\alpha^n (u)
  \right) + \lambda \partial_{\hat\db_\alpha^n (u)} F(u)
\end{equation}
for $\alpha \in \ZZ^s$, $n = n_0,\dots,n_1$. 
\eqref{eq:ApproxOptProbSol2} means that
\begin{equation}
\hat\db_\alpha^n (f)
\in \hat\db_\alpha^n (u) + \mu_{n} \, \lambda \,
\partial_{\hat\db_\alpha^n (u)} F(u),
\end{equation}
i.e., either
\begin{equation}
  \label{eq:SubdiffSoftThresh}
  \hat\db_\alpha^n (f) \in \mu_n \lambda B_1 (0)
\end{equation}
if $\hat\db_\alpha^n (u) = 0$, or
\begin{equation}
  \label{eq:SubdiffSoftThresh2}
  \hat\db_\alpha^n (f) = \left( 1 + \lambda \dfrac{\mu_n}{\left|
        \hat\db_\alpha^n (u) \right|_2} \right) \hat\db_\alpha^n (u)
\end{equation}
if $\hat\db_\alpha^n (u) \neq 0$.
Solving \eqref{eq:SubdiffSoftThresh} and \eqref{eq:SubdiffSoftThresh2}
for the coefficients of $u$, we get the explicit representation of the
solution of \eqref{eq:TVRegAlternative}. Note that this will be a
block thresholding operation and that the thresholding is in terms of
the \emph{length} of a whole \emph{vector} of wavelet coefficients.

\begin{proposition}
  The solution of \eqref{eq:ApproxOptProbSol2} can be computed by
  \emph{soft thresholding} of the \emph{length} of the normalized
  coefficient vectors 
  $\db_\alpha^n (f)$, i.e., 
  as
  \begin{equation}
    \label{eq:SoftThreshSol}
    \db_\alpha^n (u) = \left( 1 - \frac{\mu_n \lambda}{\left| \db_\alpha^n
          (f) \right|_2} \right)_{\hspace{-0.2em}+} \db_\alpha^n (f),
  \end{equation}
  for $ \alpha \in \ZZ^s$, $n=n_0,\dots,n_1$.
\end{proposition}

\begin{proof}
  The case \eqref{eq:SubdiffSoftThresh} describes
  $\hat\db_\alpha^n (u) = 0$ and is equivalent to $\left|
    \hat\db_\alpha^n (f) \right|_2 \le \mu_n{\lambda}$.
  Otherwise, \eqref{eq:SubdiffSoftThresh2} yields
  $\left| \hat\db_\alpha^n (f) \right|_2 = \left| \hat\db_\alpha^n (u)
  \right|_2 + \mu_n \lambda$ and therefore
  \begin{eqnarray}
    \hat\db_\alpha^n (f)
    & = & \left( 1 + \frac{\mu_n \lambda}{\left| \hat\db_\alpha^n (u)
          \right|_2} \right) \hat\db_\alpha^n (u) \nonumber\\
    & = & \left( 1 + \frac{\mu_n \lambda}{\left| \hat\db_\alpha^n (f)
          \right|_2 - \mu_n \lambda} \right) \hat\db_\alpha^n (u)\quad
  \end{eqnarray}
  can be easily checked to have the solution
  \begin{equation}
  \hat\db_\alpha^n (u) = \left( 1 - \dfrac{\mu_n \lambda}{\left|
        \hat\db_\alpha^n (f) \right|_2} \right) \hat\db_\alpha^n (f),
  \end{equation}
  which is \eqref{eq:SoftThreshSol}. 
\end{proof}

One could also approximate the TV norm in~\eqref{eq:TVRegProb} by a
single set of wavelet coefficients on some level~$n$ and solve
\begin{equation}
\min_u \frac12 \| f - u \|_2 + \lambda \left\| \left| \hat\db^n (u)
  \right|_2 \right\|_1
\end{equation}
by
\begin{equation}
  \label{eq:SimpleShrink}
  \db_\alpha^n (u) = \left( 1 - \frac{\lambda}{\left|
        \db_\alpha^n (f) \right|_2} \right)_{\hspace{-0.2em}+} \db_\alpha^n (f), \qquad
  \alpha \in \ZZ^s,  
\end{equation}
only applied to coefficients of level $n$. Also, the weights $\mu_n$
can be chosen arbitrarily.

To conclude, we again make a short comparison
with~\cite{steidl02:_relat_soft_wavel_shrin_total_variat_denois,welk08:_local}
where wavelet shrinkage applied to \emph{each wavelet coefficient
  separately} is related to diffusion filtering 
processes. Nevertheless, 
our findings show that for an approximation of the TV functional the
shrinkage process has to be adapted: 
\begin{enumerate}
\item The shrinkage has to be applied to the \emph{length} of the
  vectors~$\db_\alpha^n (f)$ and not to its components
  separately. Also only those coefficients have to be taken into
  account that contain a single wavelet component, as already mentioned
  at the beginning of Section~\ref{sec:TVApprox}.
\item The coefficients have to be properly renormalized and this
  renormalization depends on the level of the wavelet
  coefficients and the dimensionality of the problem.
\item It matters that we use \emph{Haar} wavelets here. The
  reconstruction of first derivatives requires, at least in our
  approach, univariate wavelets with only a \emph{single} vanishing
  moment, and the simple explicit expressions for proper
  renormalization are even due to the explicit nature and support size
  of Haar wavelets.  
\end{enumerate}
Under these conditions, we can give precise error estimates for the
procedure that works entirely on the wavelet coefficients.

\section{Applications to volume data}
\label{sec:Applications}

We finally give some numerical results obtained from applying the
method to a dataset from industrial computed tomography.

\subsection{Example dataset}
\begin{figure}
  \begin{center}
    \framebox{\includegraphics[width=\imagewidthincolumn]{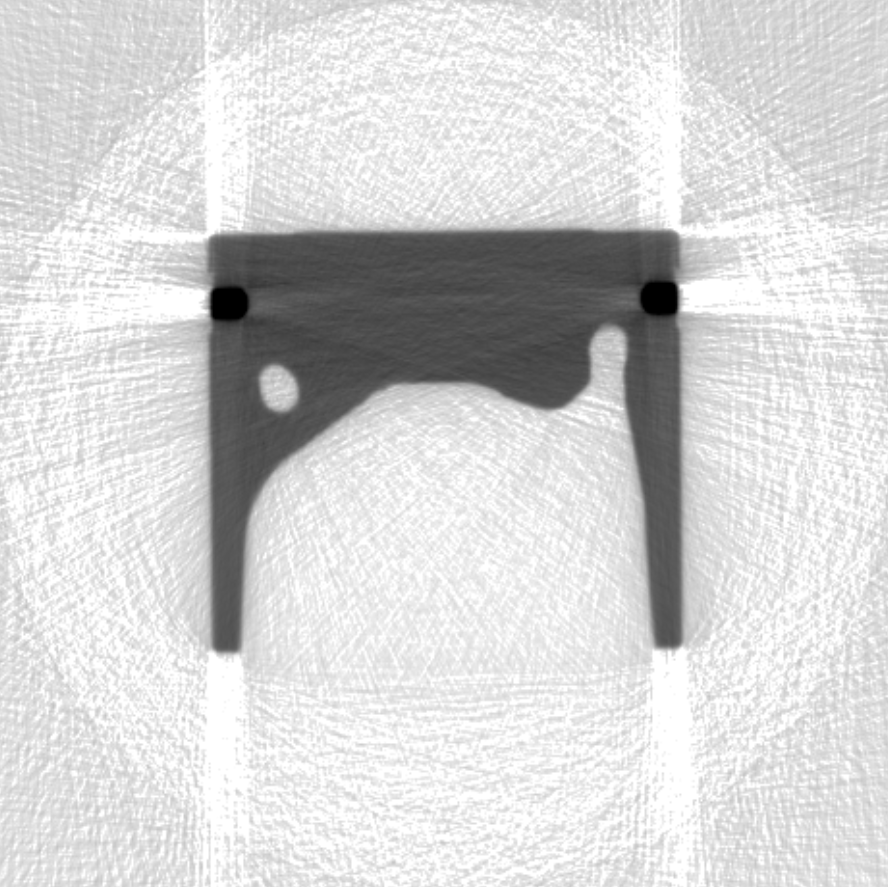}}
  \end{center}
  \caption{Slice view of a CT scan of a motor piston (courtesy of Mahle GmbH) with additional gamma correction for emphasizing the presence of noise}
  \label{applications:piston:pictures}
\end{figure}

To illustrate the results, we will use a typical industrial dataset, a Mahle
motor piston, as a running example. In Fig.~\ref{applications:piston:pictures},
we see a single 2D slice view of that data, where the different materials are
visible: surrounding (noisy) air, styrofoam (piston fixation), aluminum (piston
body) and iron (ring). The dataset itself is of size~${\num{464}\times \num{464}\times \num{414}}$
voxels and contains nonnegative $16$-bit integer values. Due to its
well-distinguishable materials and the fact that, by means of discrete
differences, we can also compute its TV norm explicitly for comparison, it is
nevertheless useful for illustration purposes despite its small size.
\begin{figure}
	\begin{center}
		\framebox{\includegraphics[width=\imagewidthincolumn]{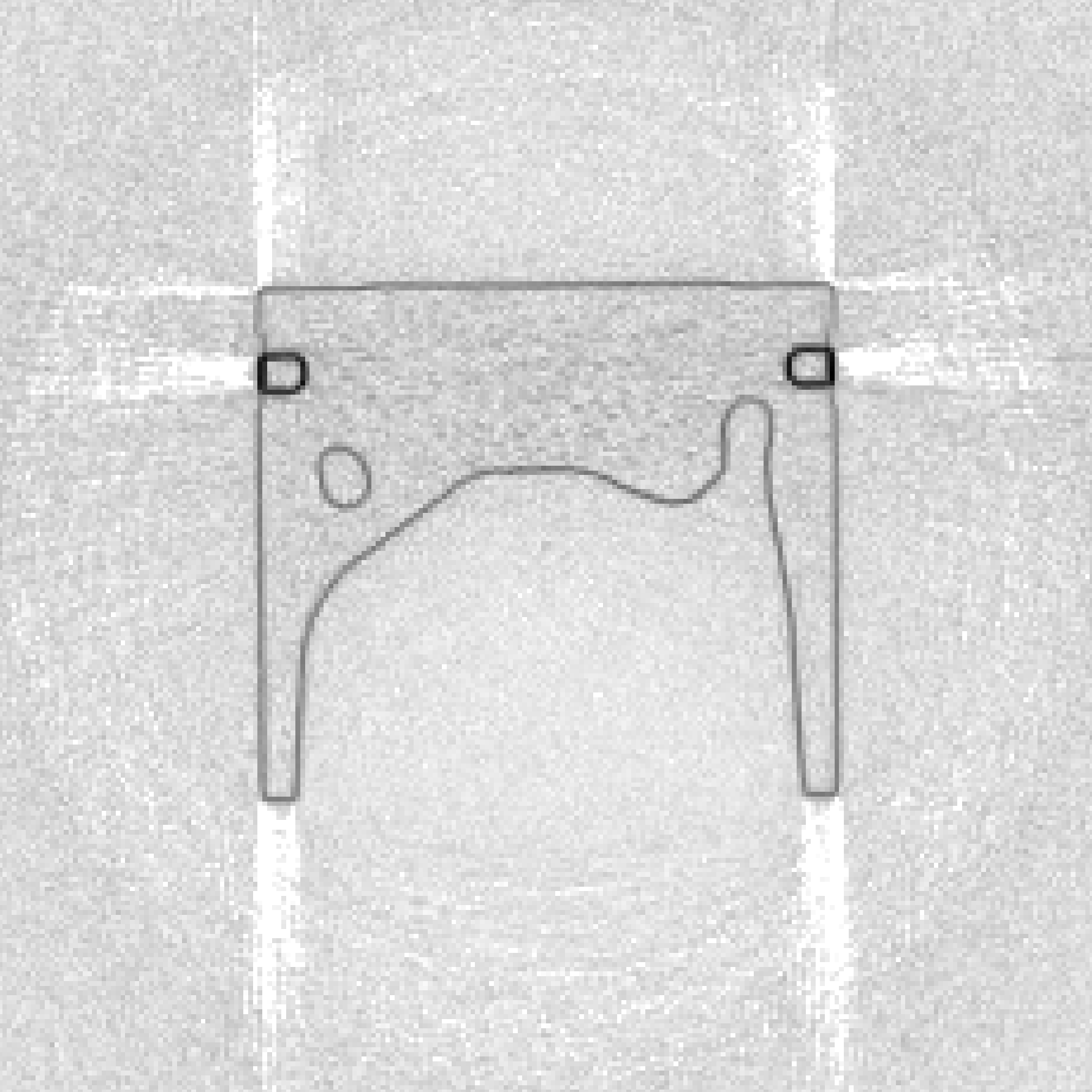}}
	\end{center}
	\caption{Slice view of the gradient magnitudes estimated from the Haar wavelet coefficients of the motor piston dataset with additional gamma correction and thresholding to balance visibility of material transitions and noise}
	\label{applications:piston:gradients}
\end{figure}
As described in Section~\ref{sec:Gradients}, the single-wavelet-component coefficients can be used for approximating the gradients, see Fig.~\ref{applications:piston:gradients}.

\subsection{Numerical examples and heuristics}
Thresholding on the wavelet coefficients with $\theta = \epsilon_j$
clearly has its limitations that will become visible when applying
very strong regularization to a dataset as the approach thresholds
only $s$ blocks of coefficients, while the other $2^s - s - 1$ ones
with $|\theta| > 1$ remain unaffected. This phenomenon becomes more
and more prominent in higher dimensions and is very well observable
for $s=3$ already.

To compensate this behavior, we propose a \emph{heuristic sparsification}
in the following way: whenever the vector $\db_\alpha^n (f)$ is thresholded
to zero for some index~$\alpha\in\ZZ^s$ and level~$n\in\NN_0$,
we set \emph{all} wavelet coefficients $d^n_{\theta,\alpha}(f)$,
$\theta\in\{0,1\}^s\setminus\{0\}$, to zero and not only $d^n_{\epsilon_j,\alpha}(f)$,
$j=1,\ldots,s$, which make up the vector $\db^n_{\alpha}(f)$. 
The rationale behind this heuristic is, on the one hand, the assumption of
locally homogeneous data combined with the observation stated in
Remark~\ref{R:CoeffDecayTheta}, that in such a situation, the wavelet
coefficients with $|\theta| > 1$ should decay faster than those with
$|\theta| = 1$. On the other hand, higher order wavelet coefficients are
usually more sensitive to noise and therefore they may be nonzero just
because of higher order reactions to noise.

\subsection{Results}
\begin{table}[t]
  \begin{tabular}{c|c|c|c|c}
    $\lambda$ & $10^2$ & $10^3$ & $10^4$ & $10^5$ \\
    \hline
    $\| u \|_{TV} / \| f \|_{TV}$ & $97\%$ & $81\%$ & $72\%$ & $76\%$ \\
    $\frac{\sum \mu_n \left\| \hat\db^n (u) \right\|_1}{\sum \mu_n \left\| \hat\db^n (f) \right\|_1}$  & $93\%$ & $49\%$ & $20\%$ & $8.5\%$ \\
    $\| f-u \|_2 / \|f\|_2$ & $0.38\%$ & $2.9\%$ & $6.5\%$ & $13\%$ \\
    PSNR & $71$ & $53$ & $46$ & $41$ \\
    Zeros & $19\%$ & $40\%$ & $52\%$ & $53\%$ \\
    \hline
  \end{tabular}
  \caption{Results (up to two digits of accuracy) of approximate TV for multiple
   regularization parameters, relative to the values for the original data,
   e.g., to $\| f \|_{TV}$. ``Zeros'' stands for wavelet
    coefficients that have been thresholded to zero}
  \label{applications:piston:experiments:livetv}
\end{table}
\begin{table}[t]
  \begin{tabular}{c|c|c|c|c}
    $\lambda$ & $10^2$ & $10^3$ & $10^4$ & $10^5$ \\
    \hline
    $\| u \|_{TV} / \| f \|_{TV}$ & $97\%$ & $79\%$ & $55\%$ & $53\%$\\
    $\frac{\sum \mu_n \left\| \hat\db^n (u) \right\|_1}{\sum \mu_n \left\| \hat\db^n (f) \right\|_1}$  & $93\%$ & $49\%$ & $20\%$ & $8.5\%$ \\
    $\| f-u \|_2 / \|f\|_2$ & $0.41\%$ & $3.4\%$ & $7.3\%$ & $13\%$ \\
    PSNR & $70$ & $52$ & $45$ & $40$ \\
    Zeros & $20\%$ & $69\%$ & $98\%$ & $99.8\%$ \\
    \hline
  \end{tabular}
\caption{Results (up to two digits of accuracy) obtained by also zeroing
  coefficients with $|\theta| > 1$ if all coefficients with $|\theta| = 1$
  are thresholded to zero, affecting all metrics except the Haar-wavelet-based
  TV approximation. Note that the behavior of the TV norm is now more similar
  to the behavior of the TV approximation}
\label{applications:piston:experiments:sparsetv}
\end{table}
Given the motor piston dataset~$f$, we compare the solutions to the approximate TV minimization problem~$u$ for different parameters~$\lambda$ satisfying~\eqref{eq:SoftThreshSol} in Tables~\ref{applications:piston:experiments:livetv}
and~\ref{applications:piston:experiments:sparsetv}. The used metrics consider data fidelity, regularization performance, and memory efficiency, also normalized with respect to the corresponding values of the original data~$f$.
The proposed method does indeed lower the approximate wavelet TV
norm. However, it seems to not influence all the coefficients
necessary to reduce the TV norm in the standard basis.
The heuristic sparsification, on the other hand, shows the same reduction of the
approximate wavelet TV norm but also consistently leads to a smaller TV
norm in the standard basis, and a higher overall percentage of coefficients of value zero as well.
\begin{figure}
  \begin{center}
    \framebox{\includegraphics[width=\imagewidthincolumn]{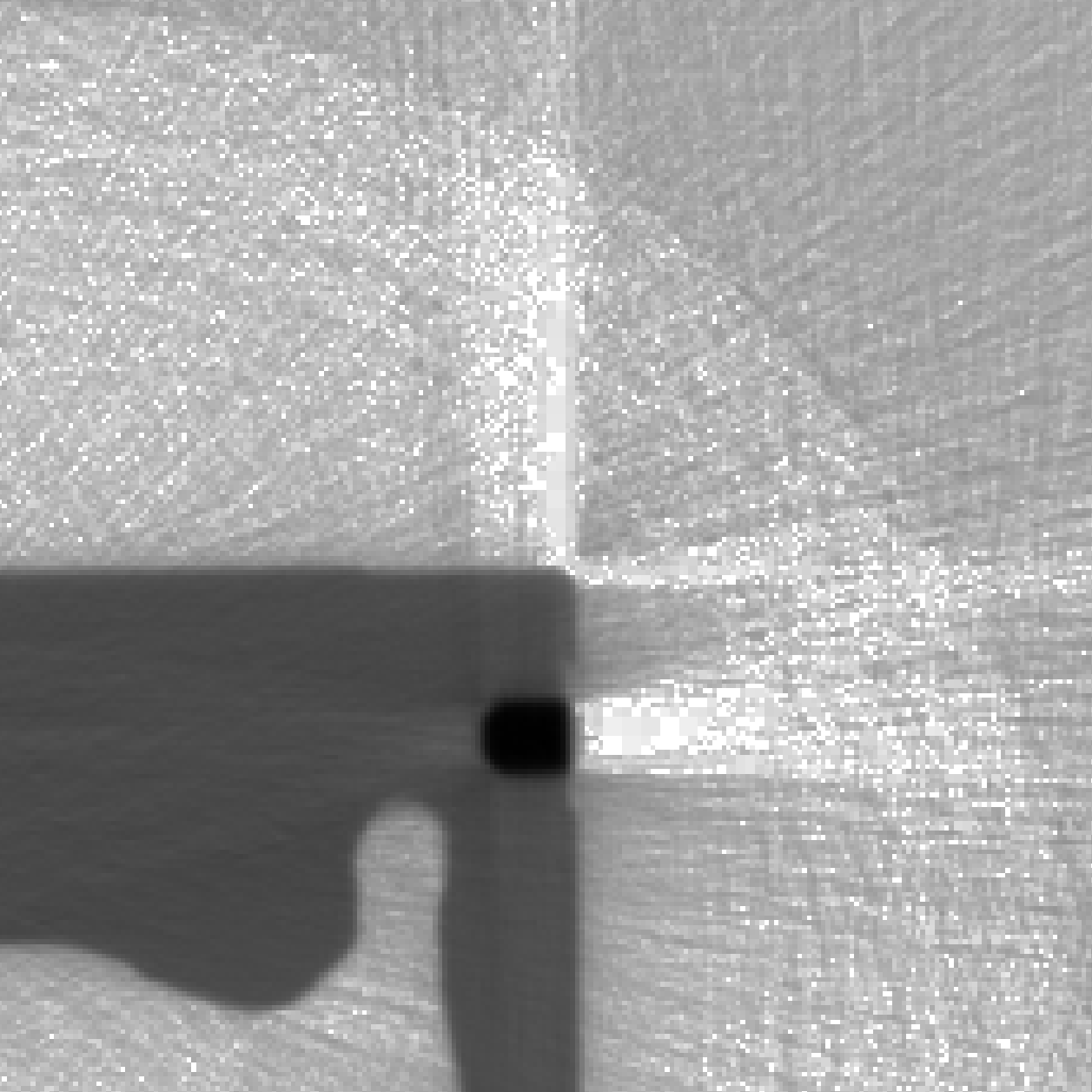}} \imagelinebreak
    \framebox{\includegraphics[width=\imagewidthincolumn]{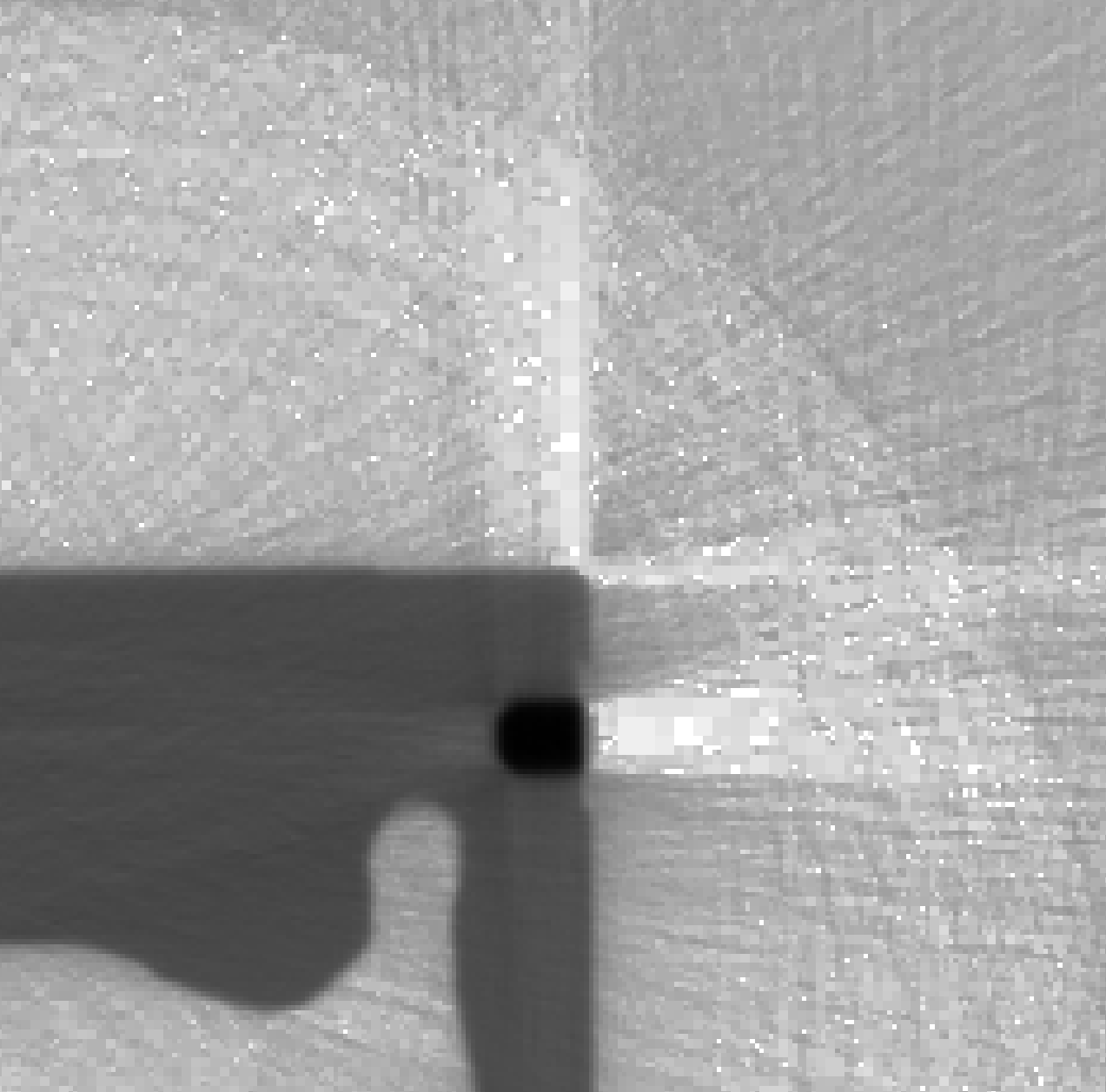}}
  \end{center}
  \caption{\emph{Top}: proposed method. \emph{Bottom}: including heuristic sparsification. In both cases, the parameter~$\lambda=10^3$ is used and only a part of the center XY slice is shown. Gamma correction was applied to highlight the noise details}
  \label{applications:piston:experiments:sparsetv-versus-livetv-lambda10e3}
\end{figure}
\begin{figure}
\begin{center}
  \framebox{\includegraphics[width=\imagewidthincolumn]{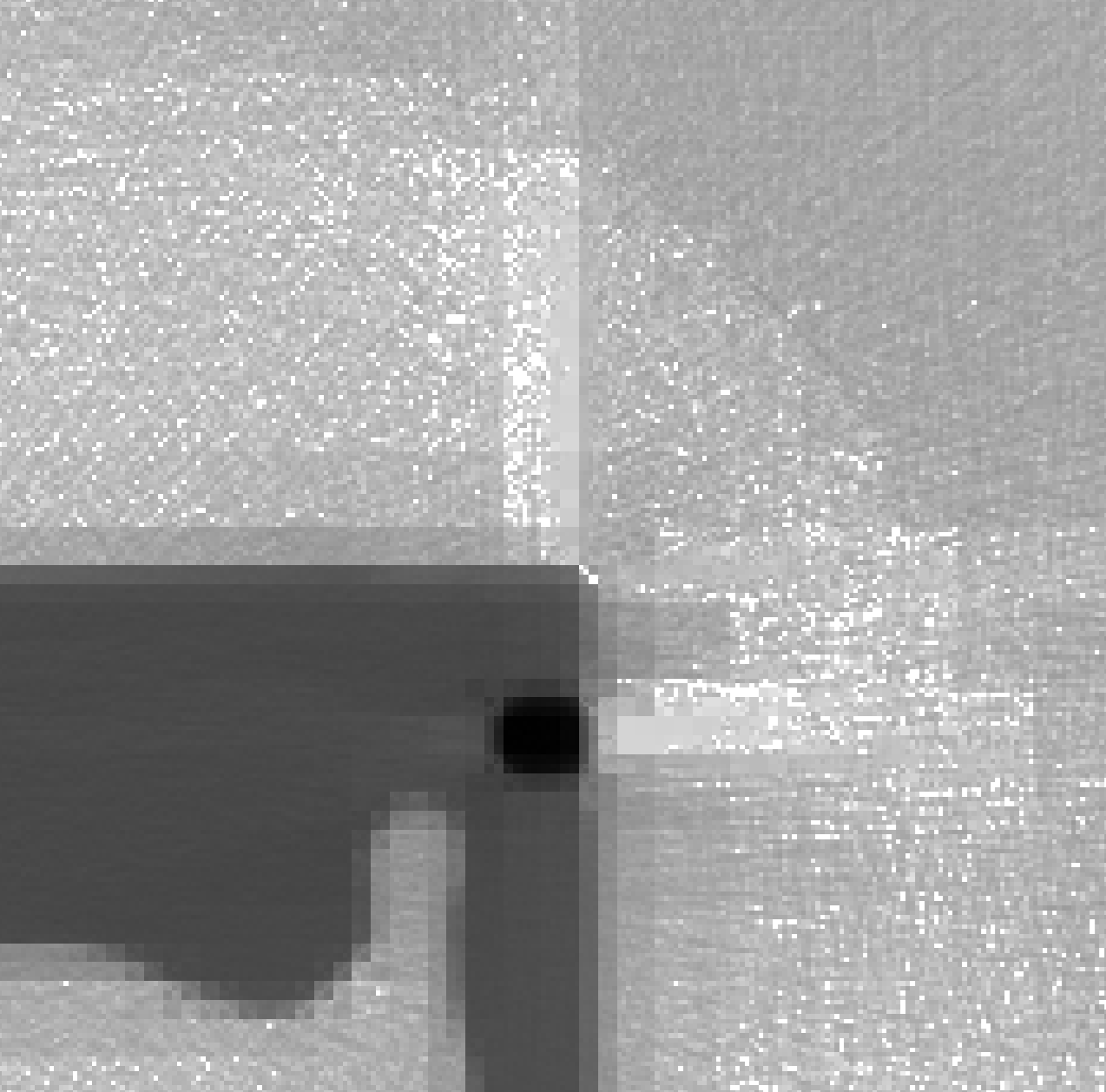}} \imagelinebreak
  \framebox{\includegraphics[width=\imagewidthincolumn]{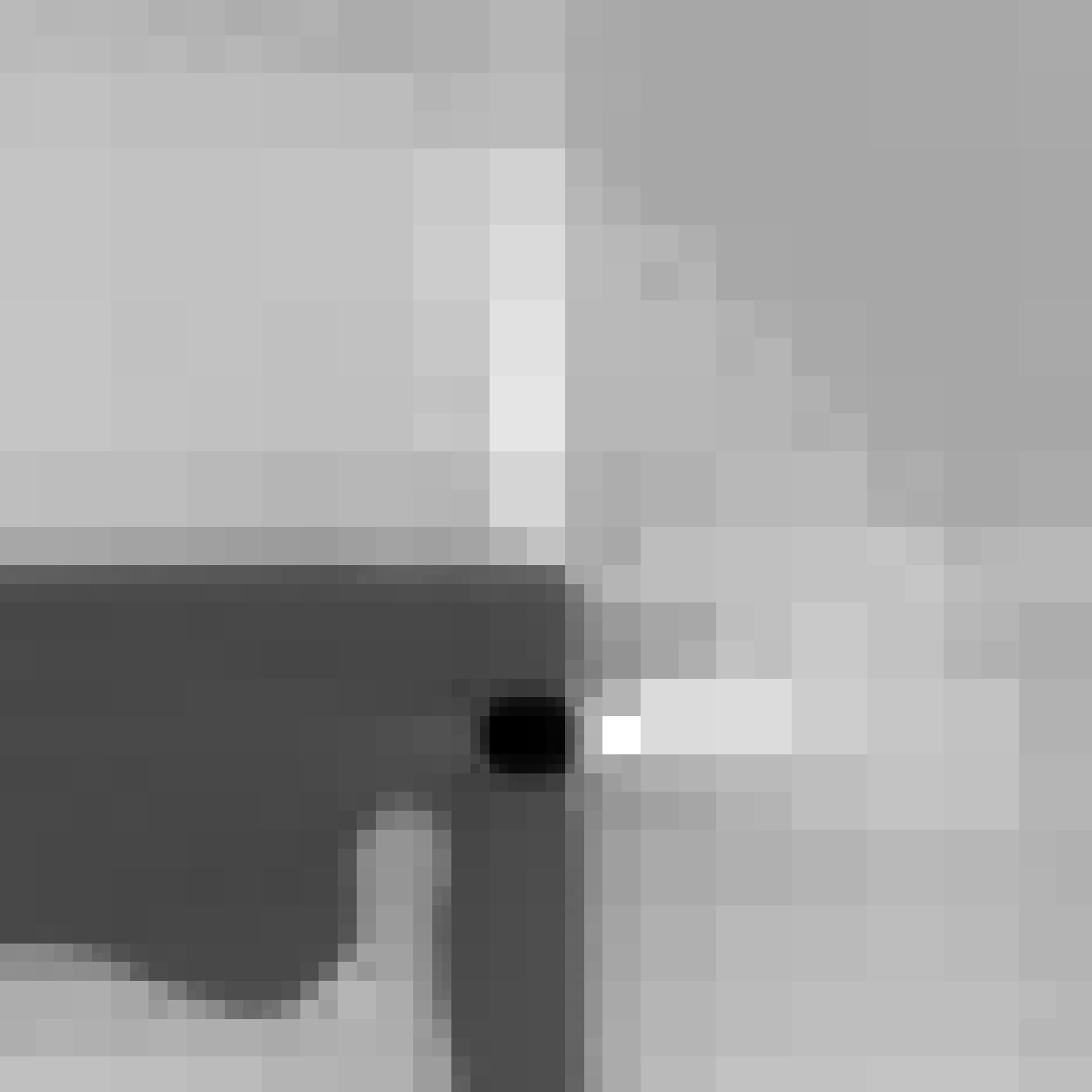}}
\end{center}
\caption{\emph{Top}: proposed method. \emph{Bottom}: including heuristic sparsification. In both cases, the parameter~$\lambda=10^5$ is used and only a part of the center XY slice is shown. Gamma correction was applied to highlight the noise details}
\label{applications:piston:experiments:sparsetv-versus-livetv-lambda10e5}
\end{figure}
Fig.~\ref{applications:piston:experiments:sparsetv-versus-livetv-lambda10e3} and~\ref{applications:piston:experiments:sparsetv-versus-livetv-lambda10e5} indicate that, for larger thresholds, the reduction of noise texture is stronger when using the heuristic sparsification while the image quality at material transitions is comparable.
The three-dimensional views show similar noise texture properties.

\begin{figure}
	\begin{center}
		\framebox{\includegraphics[width=\imagewidthincolumn]{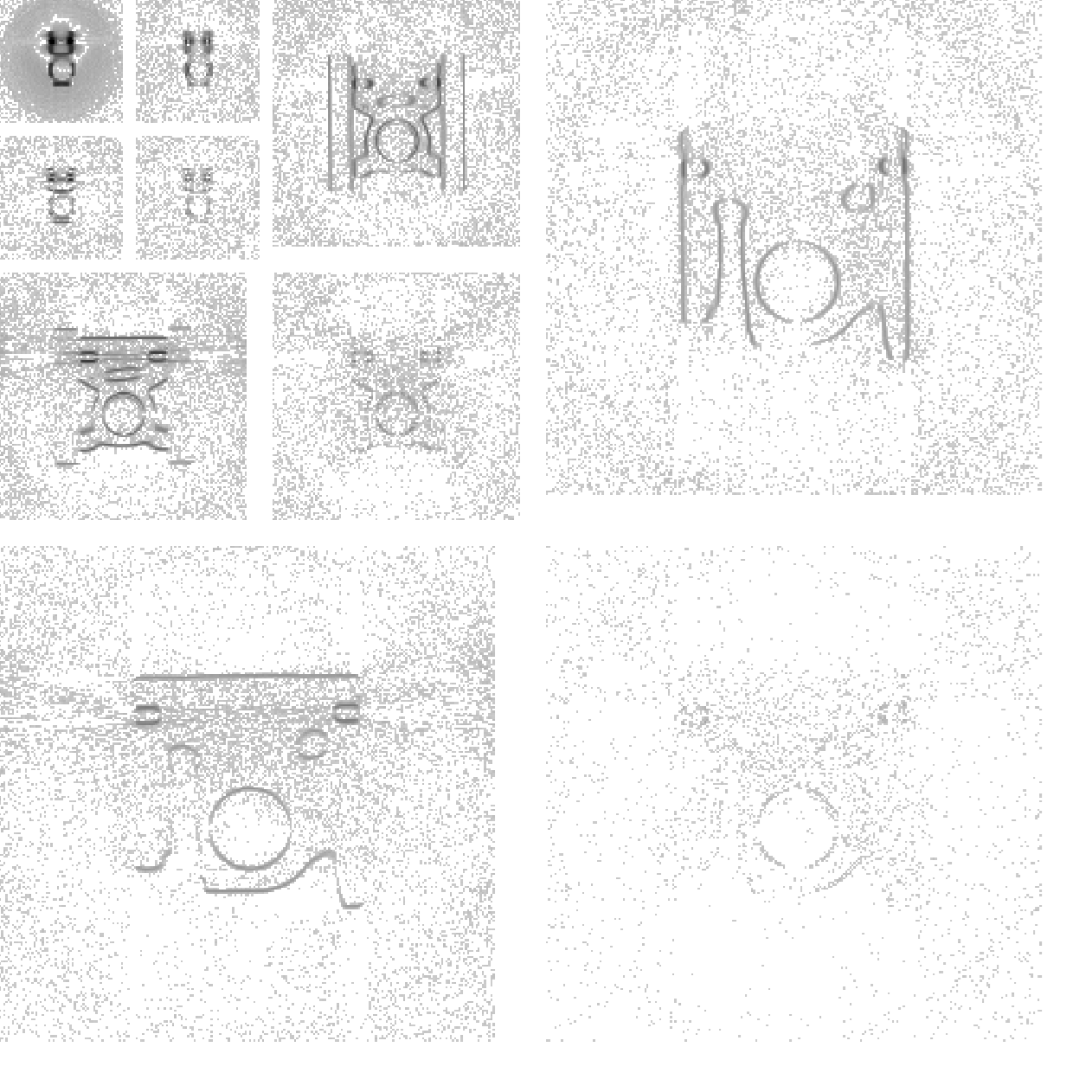}} \imagelinebreak
		\framebox{\includegraphics[width=\imagewidthincolumn]{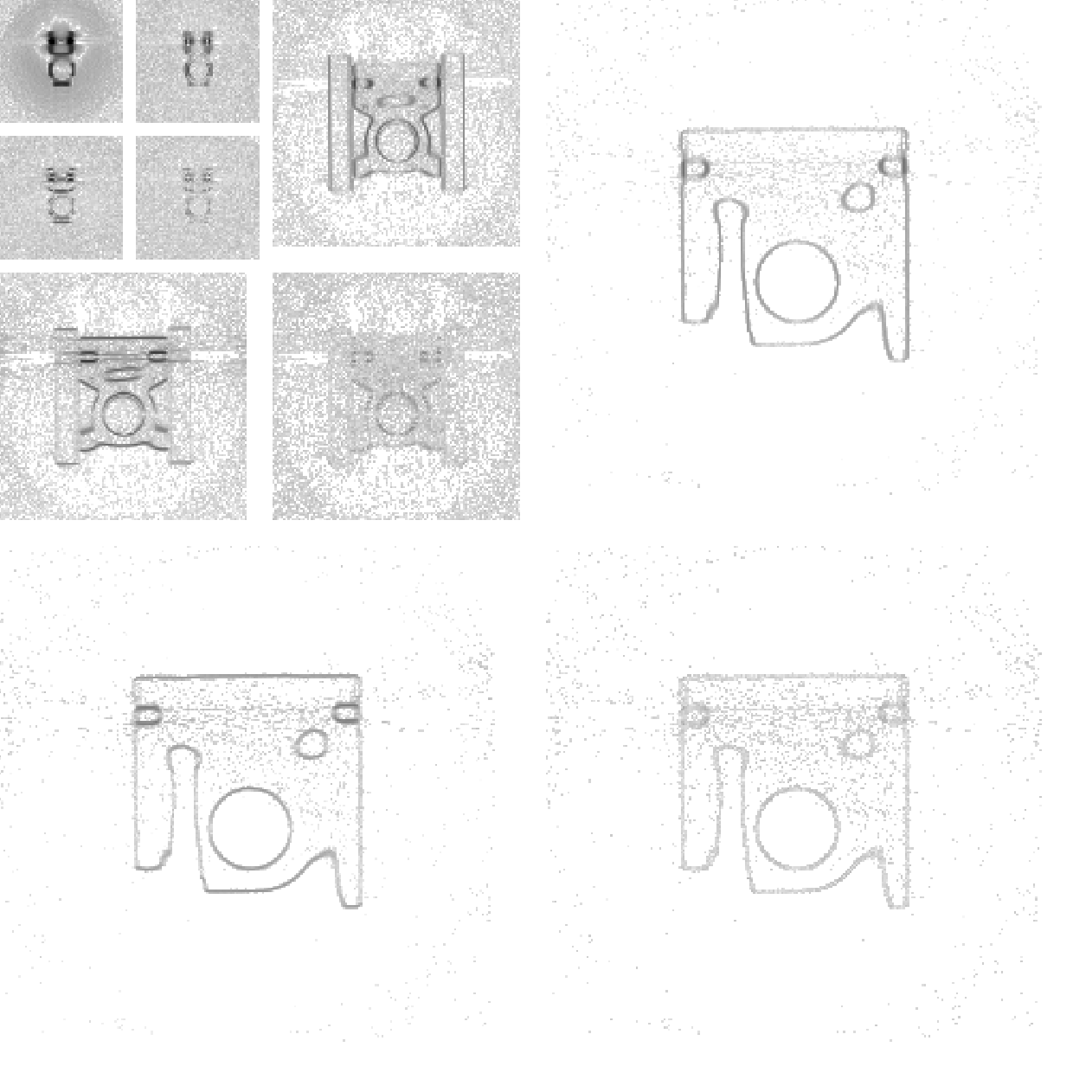}}
	\end{center}
	\caption{Slice through absolute wavelet coefficients of the motor piston dataset after soft thresholding (top) and after regularization via the heuristic sparsification (bottom), removing more noise surrounding the piston. Furthermore, all edges are equally present for all wavelet functions in all directions. Gamma correction was applied to both images to highlight the noise details}
	\label{applications:piston:experiments:sparsetv-versus-hardthresholding-coefficients}
\end{figure}
\begin{figure}
  \begin{center}
    \framebox{\includegraphics[width=\imagewidthincolumn]{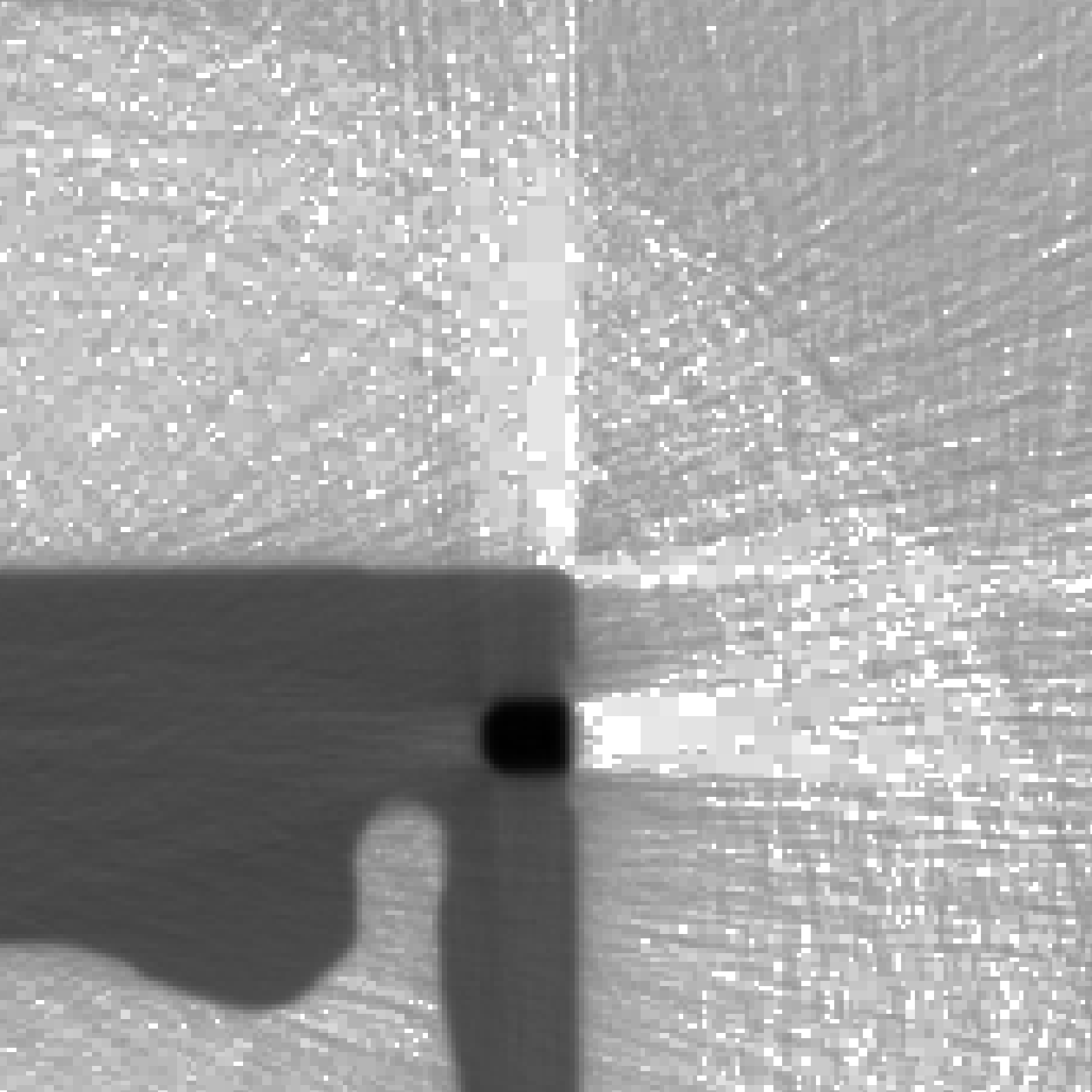}} \imagelinebreak
    \framebox{\includegraphics[width=\imagewidthincolumn]{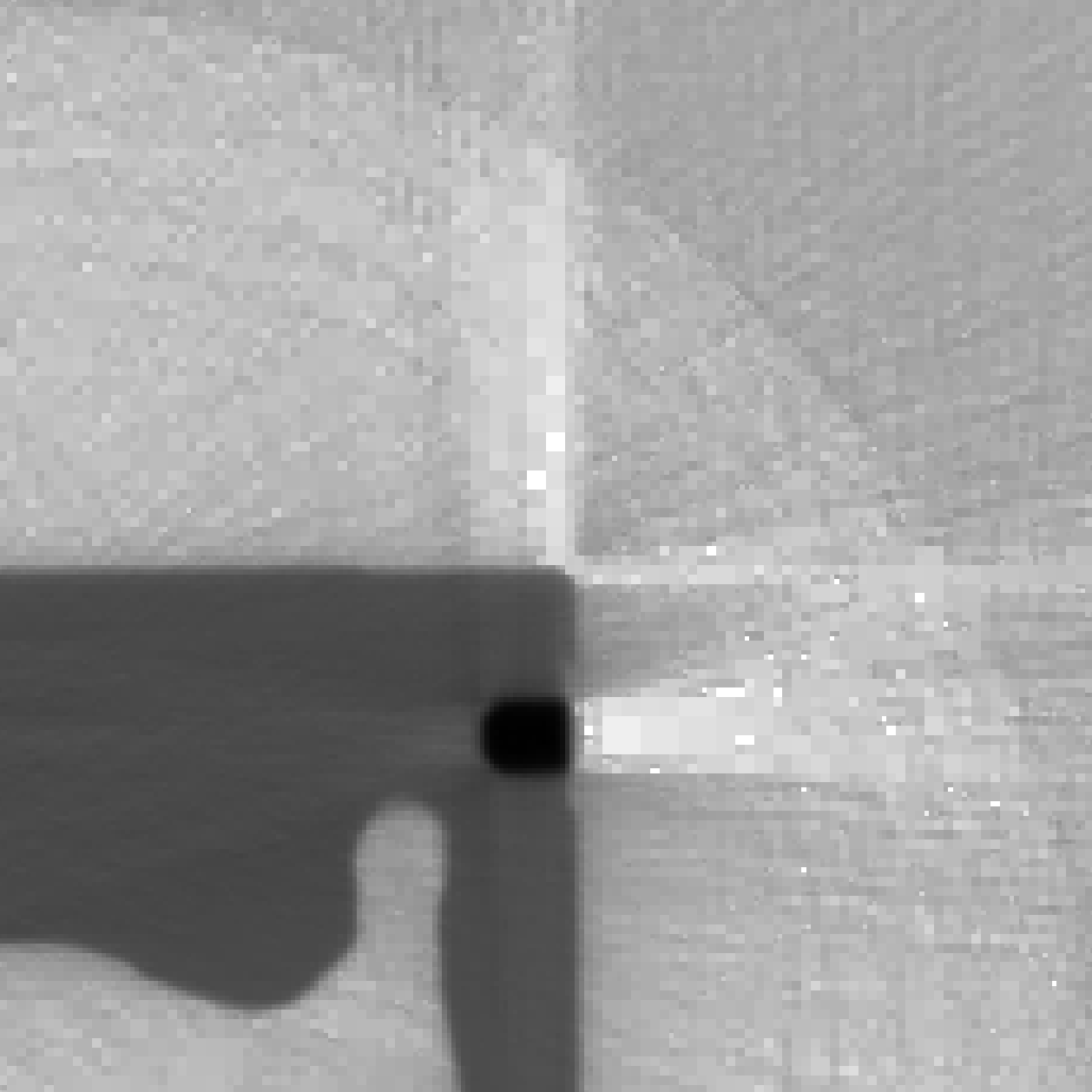}}
  \end{center}
  \caption{Reconstructed motor piston slice region after soft thresholding (top), and after heuristic sparsification (bottom), exhibiting less speckle-like noise in different resolutions, and similar sharpness at the \emph{important} edges at the borders between air and aluminum. 
  	Both reconstructions are based on the wavelet data that is visualized in Fig.~\ref{applications:piston:experiments:sparsetv-versus-hardthresholding-coefficients}. Gamma correction was applied to both images to highlight the noise details}
  \label{applications:piston:experiments:sparsetv-versus-hardthresholding}
\end{figure}

In this practical example, the heuristic sparsification also leads to less speckle-like phenomena regarding the surrounding noise compared to soft thresholding, where the coefficients are manipulated independently, which can be seen in Fig.~\ref{applications:piston:experiments:sparsetv-versus-hardthresholding-coefficients} and~\ref{applications:piston:experiments:sparsetv-versus-hardthresholding}.
The additional heuristic sparsification addresses this kind of discontinuous behavior, which is typical for componentwise thresholding of tensorized Haar wavelet coefficients, by considering \emph{all} wavelet coefficients $d^n_{\theta,\alpha}(f)$, $\theta\in\{0,1\}^s\setminus\{0\}$, at once for some index~$\alpha\in\ZZ^s$ and level~$n\in\NN_0$ instead.

\subsection{Timing}
\begin{table}[t]
  \begin{tabular}{c|c|c|c|c}
    $\lambda$ & $10^2$ & $10^3$ & $10^4$ & $10^5$ \\
    \hline
    min. time (s) & 0.086 & 0.036 & 0.025 & 0.034 \\
    avg. time (s) & 0.093 & 0.039 & 0.037 & 0.037 \\
    max. time (s) & 0.108 & 0.044 & 0.041 & 0.040 \\
    \hline
  \end{tabular}
\vspace*{1em}

  \begin{tabular}{c|c|c|c|c}
    $\lambda$ & $10^2$ & $10^3$ & $10^4$ & $10^5$ \\
    \hline
    min. time (s) & 0.034 & 0.036 & 0.031 & 0.031 \\
    avg. time (s) &  0.047 & 0.048 & 0.045 & 0.043 \\
    max. time (s) & 0.053 & 0.052 & 0.048 & 0.050 \\
    \hline
  \end{tabular}
\vspace*{1em}

  \begin{tabular}{c|c|c|c|c}
    $\lambda$ & $10^2$ & $10^3$ & $10^4$ & $10^5$ \\
    \hline
    min. time (s) & 0.033 & 0.047 & 0.045 & 0.041 \\
    avg. time (s) & 0.048 & 0.059 & 0.058 & 0.052 \\
    max. time (s) & 0.053 & 0.071 & 0.064 & 0.060 \\
    \hline
  \end{tabular}
  \caption{Regularization times in seconds for the piston dataset. Soft thresholding (\emph{top}), the proposed method (\emph{middle}), and the heuristic sparsification (\emph{bottom}) are applied to the wavelet coefficients with ten repetitions}
  \label{timing:piston}
\end{table}
\begin{table}[t]
  \begin{tabular}{c|c|c|c|c}
    $\lambda$ & $10^2$ & $10^3$ & $10^4$ & $10^5$ \\
    \hline
    min. time (s) & 0.342 & 0.293 & 0.297 & 0.297 \\
    avg. time (s) & 0.360 & 0.306 & 0.306 & 0.308 \\
    max. time (s) & 0.382 & 0.329 & 0.318 & 0.321 \\
    \hline
  \end{tabular}
\vspace*{1em}

  \begin{tabular}{c|c|c|c|c}
    $\lambda$ & $10^2$ & $10^3$ & $10^4$ & $10^5$ \\
    \hline
    min. time (s) & 0.464 & 0.460 & 0.457 & 0.445 \\
    avg. time (s) & 0.486 & 0.480 & 0.472 & 0.475 \\
    max. time (s) & 0.535 & 0.506 & 0.494 & 0.498 \\
    \hline
  \end{tabular}
\vspace*{1em}

  \begin{tabular}{c|c|c|c|c}
    $\lambda$ & $10^2$ & $10^3$ & $10^4$ & $10^5$ \\
    \hline
    min. time (s) & 0.574 & 0.596 & 0.580 & 0.551 \\
    avg. time (s) & 0.624 & 0.615 & 0.592 & 0.584 \\
    max. time (s) & 0.709 & 0.665 & 0.644 & 0.635 \\
    \hline
  \end{tabular}
  \caption{Regularization times in seconds for the mummy teeth ROI. Soft thresholding (\emph{top}), the proposed method (\emph{middle}), and the heuristic sparsification (\emph{bottom}) are applied to the wavelet coefficients with ten repetitions}
  \label{timing:mummy:roi}
\end{table}
Given Haar-wavelet-transformed data, we now compare the runtimes of a standard wavelet-based regularization technique, soft thresholding, with the proposed method and the heuristic sparsification.
Componentwise soft thresholding with parameter~$\lambda$ modifies a wavelet coefficient~$x\in\RR$ by calculating
\begin{equation}
x \mapsto \textnormal{sgn}(x) \left(|x| - \lambda\right)_+.
\end{equation}
The CPU-only implementation was run on a AMD Ryzen 7 4800H system.
The two datasets to test are the motor piston wavelet coefficients without any manipulation, and the \SI{30}{\giga\byte} mummy wavelet coefficients that remain after zeroing small coefficients (via hard thresholding) that are mentioned in the introduction.
For the piston, the complete dataset of size~${\num{464}\times \num{464}\times \num{414}}$ voxels is compared.
For the mummy, not the whole volume of size~${\num{7584}\times \num{7584}\times \num{9216}}$ voxels, but a region of interest of size~${\num{1024}\times \num{1024}\times \num{1024}}$ containing the teeth similar to Fig.~\ref{fig:mummyTeeth} is considered.
In both cases, regularization time was measured for ten repetitions.
Tables~\ref{timing:piston} and~\ref{timing:mummy:roi} show that, compared to pointwise soft thresholding, the proposed technique needs up to~$50\%$ more computation time, whereas the additional heuristic sparsification may take~$100\%$ more time on average.
However, note that in relation to the calculation of the inverse wavelet transform (around~24 seconds for the mummy teeth ROI), this increase in runtime is negligible and demonstrates that the proposed Haar-wavelet-based approximate TV minimization is a computationally cheap regularization strategy.

\section{Conclusion}

In contrast to computationally expensive TV regularization methods, our
approach only considers an approximation of the TV norm, and does it in an
efficient way by directly shrinking the Haar wavelet coefficients. This renders
TV-like regularization feasible for very large datasets. In two examples,
the performance of our method was discussed and compared to a standard
wavelet-based regularization approach. The results can even be improved by a
computationally cheap and straightforward heuristic modification of the thresholding process.

\backmatter

\bmhead{Acknowledgements}
This work was supported by the German Federal Ministry of Education and Research (BMBF) within the project \emph{KI4D4E: Ein KI-basiertes Framework f\"ur die Visualisierung und Auswertung der massiven Datenmengen der 4D Tomographie f\"ur Endanwender von Beamlines} under the title 05D2022 in collaboration between the Fraunhofer Gesellschaft zur F\"orderung der angewandten Forschung e.V., the University of Stuttgart, the University of Passau, the Friedrich-Alexander University of Erlangen, the Karlsruhe Institute of Technology, the Helmholtz center Berlin, the Helmholtz center Hereon, and the Forschungszentrum J\"ulich.

We thank Christoph Heinzl for providing us with Figures~\ref{fig:mummyFull} and~\ref{fig:mummyTeeth}.

\bmhead{Conflict of interest}
The authors declare that they have no conflict of interest.

\bmhead{ORCID iDs}
\begin{itemize}
	\item Tomas Sauer \url{https://orcid.org/0000-0002-3182-2141}
	\item Andreas Michael Stock \url{https://orcid.org/0000-0003-2969-3990}
\end{itemize}



\end{document}